\documentclass{article}
\usepackage{xcolor}
\definecolor{lccx}{HTML}{92268F}
\usepackage[colorlinks,citecolor=lccx]{hyperref}

\ifpdf
\hypersetup{
	pdftitle={Randomization and Performance Improvement for IOCP with TV},
	pdfauthor={},
	linkcolor=lccx
}
\fi

\usepackage[utf8]{inputenc}
\usepackage{amsfonts}
\usepackage{amsmath,amscd,amssymb,mathrsfs,setspace}
\usepackage{amsthm}
\usepackage{mathtools}
\usepackage{latexsym,epsf,epsfig}
\usepackage{thmtools}
\usepackage{thm-restate}
\usepackage{authblk}

\usepackage[a4paper,top=2cm,bottom=2cm,left=3cm,right=3cm,marginparwidth=1.75cm]{geometry}

\makeatletter
\let\cl@chapter\relax
\makeatother
\usepackage[nameinlink]{cleveref}
\usepackage{pgfplots}
\pgfplotsset{compat=newest}
\usepgflibrary{fpu}

\usepackage{enumitem}
\usepackage{makecell}
\usepackage{booktabs}
\usepackage{multirow}
\usepackage{makecell}

\usepackage[mathscr]{euscript}
\usepackage{dutchcal}
\usepackage{adjustbox}
\usepackage{algorithm}
\usepackage{algorithmicx}
\usepackage{algpseudocode}
\usepackage{verbatim}
\usepackage{caption}
\usepackage{subcaption}

\declaretheorem[name=Theorem, numberwithin=section]{theorem}

\newtheorem{lemma}[theorem]{Lemma}

\theoremstyle{definition}
\newtheorem{definition}[theorem]{Definition}

\theoremstyle{remark}
\newtheorem{remark}[theorem]{Remark}
\newtheorem{assumption}[theorem]{Assumption}

\crefname{theorem}{Theorem}{Theorems}
\Crefname{theorem}{Theorem}{Theorems}
\crefname{assumption}{Assumption}{Assumptions}
\Crefname{assumption}{Assumption}{Assumptions}
\crefname{lemma}{Lemma}{Lemmas}
\Crefname{lemma}{Lemma}{Lemmas}
\crefname{definition}{Definition}{Definitions}
\Crefname{definition}{Definition}{Definitions}
\crefname{proposition}{Proposition}{Propositions}
\Crefname{proposition}{Proposition}{Propositions}
\crefname{algorithm}{Algorithm}{Algorithms}
\Crefname{algorithm}{Algorithm}{Algorithms}
\crefname{section}{Section}{Sections}
\Crefname{section}{Section}{Sections}
\crefname{appendix}{Appendix}{Appendices}
\Crefname{appendix}{Appendix}{Appendices}

\DeclareMathOperator*{\argmax}{arg\,max}

\newcommand{\E}{\mathbb{E}}
\newcommand{\N}{\mathbb{N}}
\newcommand{\R}{\mathbb{R}}
\newcommand{\Z}{\mathbb{Z}}
\newcommand{\cA}{\mathcal{A}}
\newcommand{\cW}{\mathcal{W}}
\newcommand{\Ha}{\mathcal{H}}
\newcommand{\conv}{\operatorname{conv}}

\newcommand{\pred}{\operatorname{pred}}
\newcommand{\ared}{\operatorname{ared}}
\newcommand{\TV}{\operatorname{TV}}
\DeclareMathOperator*{\TVh}{TV^h}

\DeclareMathOperator*{\BV}{BV}
\DeclareMathOperator*{\BVW}{BV_W}

\DeclareMathOperator*{\dvg}{div}

\newcommand*\dd{\mathop{}\!\mathrm{d}}

\newcommand{\weakstarto}{\stackrel{\ast}{\rightharpoonup}}
\newcommand{\bdvg}[1]{\dvg{}_{{\hspace*{-2pt}#1\hspace*{2pt}}}}

\newcommand{\calA}{\mathcal{A}}
\newcommand{\calD}{\mathcal{D}}

\newcommand{\calW}{\mathcal{W}}
\newcommand{\calQ}{\mathcal{Q}}

\crefname{assumption}{Assumption}{Assumptions}
\Crefname{assumption}{Assumption}{Assumptions}
\crefname{lemma}{Lemma}{Lemmas}
\Crefname{lemma}{Lemma}{Lemmas}
\crefname{definition}{Definition}{Definitions}
\Crefname{definition}{Definition}{Definitions}
\crefname{proposition}{Proposition}{Propositions}
\Crefname{proposition}{Proposition}{Propositions}
\crefname{algorithm}{Algorithm}{Algorithms}
\Crefname{algorithm}{Algorithm}{Algorithms}
\crefname{remark}{Remark}{Remark}
\Crefname{remark}{Remark}{Remark}

\title{Randomization and Performance Improvement for Integer Optimal Control with Total Variation Regularization
\thanks{%
    Sandia National Laboratories is a multimission laboratory
    managed and operated by National Technology and Engineering
    Solutions of Sandia, LLC., a wholly owned subsidiary of
    Honeywell International, Inc., for the U.S.\ Department of
    Energy’s National Nuclear Security Administration under
    contract DE-NA0003525.
    This paper describes objective technical results and analysis. Any
    subjective views or opinions that might be expressed in the paper
    do not necessarily represent the views of the U.S.\ Department of
    Energy or the United States Government.}
}

\author[1]{Robert Baraldi}
\author[2]{Paul Manns}
\author[2]{Lars M\"{o}sezahl}
\author[2]{Marvin Severitt}

\affil[1]{Sandia National Laboratories, New Mexico, USA\\
	
	\textit{rjbaral@sandia.gov}}

\affil[2]{Department of Mathematics, TU Dortmund University, Dortmund, Germany\\
	\textit{paul.manns@tu-dortmund.de}}

\begin{document}
\maketitle

\begin{abstract}
Mixed-integer PDE-constrained optimization problems
are computationally challenging due to both the combinatorial
nature of integer programming as well as evaluation
of the model. Many algorithms and subsequent performance improvements
have been developed to solve these problems, but they are often 
limited by problem size.
We numerically analyze two such algorithms: \ref{alg:slip}
and \ref{alg:patch_slip}, which solve trust-region subproblems over either
the full or partial domain, respectively. We additionally propose and
prove convergence of a randomized third algorithm, \ref{alg:randomized_patch_slip}, 
which solves trust-region subproblems over randomly selected 
patches of the domain. 
We compare performance of all
three algorithms with various improvement
techniques found throughout the literature; the purpose of this work is
to document the best combinations of these improvements in conjunction with 
various solvers. We additionally establish benchmark problems in image denoising and cloaking. 
Computational results are reported on
combinations of algorithms and improvements, and 
code used in the experiments is provided as a package.


\textbf{MSC Subject Classes}\enskip 49M27; 49M37; 65K05; 65K10; 90C11; 90C30; 93-08
\end{abstract}

\section{Introduction}\label{sec:introduction}
We are interested in solving the following class of integer optimal control 
problems \cite{baraldi2025domain,manns2023on}
\begin{gather}\label{eq:p}
\begin{aligned}
\min_{w \in L^1(\Omega)}\ 
& J(w) \coloneqq F(w) + \alpha \TV(w) \\
\text{s.t.}\quad 
& w(x) \in W \coloneqq \{w_1,\ldots,w_M\} \subset \Z
\text{ for almost every (a.e.) }
x \in \Omega,
\end{aligned}\tag{P}
\end{gather}
where $\Omega \subset \R^d$, $d \in \N$, is a bounded Lipschitz domain,
$\alpha > 0$, and $M \in \N$. The functional $F \colon L^1(\Omega) \to \R$
is lower semicontinuous with respect to convergence in $L^p(\Omega)$
for some $p \ge 1$ and bounded below, whereas
the functional $\TV \colon L^1(\Omega) \to [0,\infty]$ is the total variation seminorm \cite{ambrosio2000functions}.
In particular, we consider $F$ to be the reduced
cost functional \cite{troltzsch2010optimal}, that is,
$F = j \circ S$, where with some fidelity term $j$ and some
solution operator $S$ of a PDE that governs the dynamics underlying
\eqref{eq:p}; these are described in more detail for specific problems. 
Integer optimal control problems are a well-studied and versatile
modeling tool with many applications, see, e.g.,
\cite{gerdts2005solving,hante2017challenges,leyffer2021convergence,yu2021multidimensional}. 
The addition of a total variation term either through a
budget constraint or as a regularizing penalty in integer optimal control 
problems has recently gained substantial scientific interest 
over a large variety of settings
see, e.g.,
\cite{baraldi2025domain,buchheim2024extended,kaya2020optimal,manns2023on,manns2025discrete,marko2023integer,marko2025vector,schiemann2025discretization,severitt2023efficient,wachsmuth2024optimal,yang2025augmentation}.

Using geometrical variational analysis, \eqref{eq:p} gives rise to
trust-region algorithms with guaranteed convergence properties
that solve combinatorial subproblems
\cite{leyffer2022sequential,manns2024convergence,manns2023on}.
For $d = 1$, efficient dynamic programming-type algorithms for subproblems
\cite{marko2023integer,severitt2023efficient} exist and the convex hull
of the solution set has been characterized in \cite{buchheim2023polytope}.
The problem class is more challenging for $d \ge 2$, where---to this date and to
the best of our knowledge---no efficient subproblem solver has been found, 
NP-completeness of the subproblems is an open question
\cite{severitt2023efficient},
the discretization of the $\TV$-seminorm poses additional challenges
in the presence of the discreteness constraint \cite{schiemann2024integer},
and branch-and-bound algorithms for the subproblems
may have prohibitively long compute times already on medium-sized
instances \cite{manns2025discrete,schiemann2024integer}.
This necessitates methods that can scale the algorithms to larger problem
instances, while simultaneously integrating a correct discretization
of the total variation seminorm. Moreover, such methods
must be rigorously tested on established and pertinent
benchmark problems.

In light of discrepancy, the authors of \cite{baraldi2025domain} incorporate
domain decomposition into a trust-region algorithm to
solve many small and relatively inexpensive subproblems 
instead of a few very hard subproblems. Additionally, \cite{baraldi2025domain} 
extends the stationarity concept for \eqref{eq:p} from \cite{manns2023on}, which is 
based on so-called local variations of the level sets of (locally)
optimal points. This concept can be localized to patches (or subdomains)
that constitute an open cover of the computational domain $\Omega$
and thus have nontrivial overlap. Consequently, a point is stationary
if and only if it is stationary on all of the patches. This is in 
analogy to unconstrained optimization in Euclidean spaces, where 
stationarity (the gradient being equal to zero) is equivalent to 
stationarity per coordinate. Similar to coordinate-descent algorithms
for unconstrained optimization in Euclidean spaces, \cite{baraldi2025domain}
developed a convergence analysis for its decomposed optimization algorithm via a greedy 
patch selection for the next iterate. 
Moreover, substantial performance gains were 
achieved for relevant parameter settings and a moderate amount
of patches in \cite{baraldi2025domain}.
\paragraph{Contributions} 
Our overall goal is to compare three algorithms' computational 
performance over a variety of different settings. 
These three algorithms encompass two previously mentioned
works and a third, randomized version which we introduce
in this work: 
1) \ref{alg:slip}, the trust-region algorithm for solving \eqref{eq:p} proposed
as Algorithm 5.1 in \cite{manns2023on}; 
2) \ref{alg:patch_slip}, Algorithm 4.1 in \cite{baraldi2025domain}, which
corresponds to SLIP with trust-region subproblems defined on patches with
a greedy patch-selection mechanism; 
and 3) \ref{alg:patch_slip},  which corresponds
to \ref{alg:slip} with trust-region subproblems defined on patches with
a random patch-selection mechanism.
\ref{alg:slip} improves computational performance
of the algorithmic frameworks from \cite{baraldi2025domain,manns2023on}.
Specifically, we overcome the need of an expensive greedy approach
for convergence guarantees via this randomized algorithm. Its design is
motivated by randomized coordinate-descent methods, typical in
data science applications. We prove convergence of \ref{alg:randomized_patch_slip}
to stationary points in expectation in function space.

Those algorithms have comparable convergence properties in the sense that
the iterates produced by \ref{alg:slip} and \ref{alg:patch_slip} converge to stationary points,
see \cite{baraldi2025domain,manns2023on}. Likewise, the iterates produced by
\ref{alg:randomized_patch_slip} converge to stationary points a.s., see \cref{thm:main_convergence_theorem}.
In order to obtain a fair computational comparison below, we first present several
avenues for practical computational performance improvements for the trust-region
algorithms and the trust-region subproblem solver that operates on discretizations
of the subproblems \eqref{eq:trp}.

These accelerations of
the algorithm(s) are: an advancement of the cutting plane
generation strategy from \cite{manns2025discrete} in
the branch-and-bound-based subproblem solution; 
a novel
discretization of the $\TV$-seminorm on the subproblem
level \cite{schiemann2025discretization}; 
replacement of the alternating
solution algorithm from \cite{schiemann2025discretization} by an integrated one 
based on so-called lazy constraints;
% that successively augment the integer optimization problem
% during the solution process; 
different domain decompositions for \ref{alg:patch_slip} and
\ref{alg:randomized_patch_slip} \S\ref{sec:domain_decomposition_mechanisms}; and
a heuristic initialization method for the initial iterate $w^0$
in \S\ref{sec:heuristic_initialization}. 
% Afterwards, we present 
% the discretization of \eqref{eq:trp} and several options to improve the performance
% of the subproblem solver for the them in \cref{sec:trp_improvement}.
In addition, we provide four computational benchmark problems, 
on which we assess the effect of our algorithmic advancements.
%(in SLIP, we always have $D = \Omega$).
In our overall comparison, we run \ref{alg:slip}, 
\ref{alg:patch_slip}, \ref{alg:randomized_patch_slip}
against each other with the most successful
improvements switched on for subsequent sections.

The remainder of the paper is structured as follows. 
In \cref{sec:pre}, we introduce the concept of 
stationarity and summarize previous results 
for algorithm convergence. 
\Cref{sec:benchmark_problems} discusses
the problems upon which we test our algorithms. 
Subsequently, \cref{sec:performance_improvements_for_slip}
discusses performance improvement techniques for \ref{alg:slip}.
While noting that these extend to its subsidiaries, 
we additionally state improvements specific to \ref{alg:patch_slip}
in \cref{sec:patch_slip} and test it against 
\ref{alg:slip}. 
We develop a randomized version of \ref{alg:patch_slip} in 
\cref{sec:randomized_patch_slip}, and as before, test it against its
predecessors. 
Finally, we conclude in \cref{sec:con}
and leave some of the theory of \ref{alg:randomized_patch_slip} to
the appendix. 

\section{Notation \& Preliminaries}
\label{sec:pre}
We briefly introduce notation that is from geometric measure theory
used for convergence discussions. 
For a measurable set $E \subset \Omega$, we define its perimeter by
$P(E,\Omega) = \TV(\chi_E)$, where $\chi_E$ denotes the $\{0,1\}$-valued
indicator function of $E$. Sets with $P(E,\Omega) < \infty$ are
called \emph{Caccioppoli} sets. If $\{E_i : i \in I\}$ for some
$I \subset \N$ is a partition of $\Omega$ with
$\sum_{i \in I} P(E_i, \Omega) < \infty$, then it is a
\emph{Caccioppoli partition}.
Feasible points of \eqref{eq:p} with finite objective value are
functions in $\BV(\Omega)$, the space of functions of bounded
variation \cite{ambrosio2000functions} on $\Omega$ that
attain values in the finite set $W$ only. We abbreviate this
space as $\BVW(\Omega) \coloneqq \{ w \in \BV(\Omega) : v(x) \in  w
\text{ for a.e.\ }x \in  \Omega \}$.
The derivatives $f'(x) \in (L^p(\Omega))^*$ of
differentiable functions  $f \colon L^p(\Omega) \to \R$ with
$p \in [1,\infty)$ have representatives in $L^{p'}(\Omega)$
for the Hölder conjugate index $p'$, which we denote by $\nabla f(x)$.
First, we briefly recall the stationarity concept from \cite{manns2023on}
and its equivalent localized characterization from \cite{baraldi2025domain};
it is required for the convergence of \cref{alg:randomized_patch_slip},
introduced in \cref{sec:randomized_patch_slip} and analyzed
in \cref{sec:algorithm_analysis}. We also provide a set of standing 
assumptions and resulting existence of solutions to \eqref{eq:p},
and then state our trust-region subproblem class since it is the core building
block of the three following algorithms.
\begin{definition}[Stationarity, Definition 4.4 in \cite{manns2023on}]\label{dfn:stationarity}
	Let $F \colon L^1(\Omega) \to \R$ be continuously
	Fr\'{e}chet differentiable.
	Let $w \in \BVW(\Omega)$, that is $w = \sum_{i=1}^M w_i \chi_{E_i}$
	for some Caccioppoli partition $\{E_1,\ldots,E_M\}$  of
	$\Omega$. Let $\nabla F(w) \in C(\bar{\Omega})$.
	Then, $w$ is \emph{stationary} if
	\begin{gather}\label{eq:stationarity}
	\sum_{i=1}^{M - 1} \sum_{j=i + 1}^M
	\int_{\partial^*{E}_i \cap \partial^* E_j}
	(w_j - w_i)\nabla F(w)(x)\phi(x)\cdot n_{E_i}(x)
	- \alpha |w_i - w_j| \bdvg{E_i} \phi(x)
	\dd \Ha^{d-1}(x) = 0
	\end{gather}
	holds for all $\phi \in C^\infty_c(\Omega, \R^d)$.
\end{definition}

\begin{assumption}\label{ass:calD_open_cover}
	Let $\calD \subset 2^{\Omega}$ be a finite, open cover of $\Omega$.
\end{assumption}

\begin{definition}[Patch-stationarity,
	Definition 3.4 in \cite{baraldi2025domain}]\label{dfn:patch_stationarity}
	Let $\calD \subset 2^{\Omega}$, $F \colon L^1(\Omega) \to \R$, and $w \in \BVW(\Omega)$
	satisfy the assumptions of \cref{dfn:stationarity} and \cref{ass:calD_open_cover}.
	Then, $w$ is \emph{patch-stationary with respect to $\calD$} if for all $D \in \calD$
	the identity \eqref{eq:stationarity} holds for all $\phi \in C^\infty_c(D, \R^d)$.
	If $w$ satisfies \eqref{eq:stationarity} for all $\phi \in C^\infty_c(D, \R^d)$
	some (fixed) $D \in \calD$, we say that $w$ \emph{is patch-stationary on} $D$.
\end{definition}

As mentioned before, stationarity and patch-stationarity are equivalent in this setting,
as shown below.
\begin{theorem}[Theorem 3.5 in \cite{baraldi2025domain}]
Let $\calD \subset 2^{\Omega}$, $F \colon L^1(\Omega) \to \R$, and $w \in \BVW(\Omega)$
satisfy the assumptions of \cref{dfn:patch_stationarity}. Then $w$
is stationary if and only if it is patch-stationary with respect to $\calD$.
\end{theorem}

We now provide sufficient conditions on $F$ required
for \ref{alg:randomized_patch_slip}'s convergence analysis to hold. Specifically, we 
may make the same choice as for the (deterministic) 
patch-based trust-region algorithm in 
\cite{baraldi2025domain}. We note that they are very 
similar to the ones imposed in closely 
related works; see also Assumption 2.2 in 
\cite{manns2023homotopy} and Assumption 4.3 in
\cite{manns2023on} and the discussions of the assumptions therein.

\begin{assumption}[Assumption 4.1 in \cite{baraldi2025domain}]\label{ass:standing}~
\begin{itemize}
\item Let $F \colon L^1(\Omega) \to \R$ be bounded
below.
\item Let $F \colon L^1(\Omega) \to \R$ be continuously Fr\'echet differentiable.
\item Let $\nabla F \colon L^1(\Omega) \to  L^\infty(\Omega)$ be Lipschitz continuous on the feasible set, that is,
\[ \infty > 
L_{\nabla F}
\coloneqq 
\sup \left\{ \frac{\|\nabla F(w) - \nabla F(v)\|_{L^\infty}}{\|w - v\|_{L^1}}
: \begin{array}{r}
w(x), v(x)\in \conv W \\
\text{for a.e.\ } x \in \Omega
\end{array}
\right\}.
\]
\item $\nabla F(w) \in C(\bar{\Omega})$ for all $w \in \BVW(\Omega)$.
\end{itemize}
\end{assumption}

The first item was already stated at the beginning
of \cref{sec:introduction} and is repeated for
convenience
The fourth item of \cref{ass:standing} is not present 
in Assumption 4.1 in \cite{baraldi2025domain} but 
assumed on designated points (iterates and limit 
points of the iterations) only in
\cite{baraldi2025domain} because this extra 
regularity is only required at specific points
and not globally. The assumption can therefore be 
interpreted as a constraint qualification. We have 
decided to assume it globally for convenience
and because the designated points, where it is 
required, are not known a priori and so it makes 
sense to verify it globally anyway if it needs 
be verified.

The first two items of \cref{ass:standing}
directly imply the existence of solutions to \eqref{eq:p}.

\begin{theorem}[Proposition 2.3
in \cite{leyffer2022sequential}]
Let \cref{ass:standing} hold. Then
\eqref{eq:p} admits a minimizer.
\end{theorem}
Finally, we introduce our trust-region subproblem class that is serves as the core building block for all of the algorithms
in this article. For given linearization point $\bar{w} \in \BVW(\Omega)$, gradient (approximation)  $g$, patch (subdomain) $D \subset \Omega$,
and trust-region radius $\Delta$, it is defined as:
\begin{gather}\label{eq:trp}
\operatorname{\ref{eq:trp}}(\bar{w}, g, D, \Delta) \coloneqq
\left\{
\begin{aligned}
\min_{w \in L^1(\Omega)}\ & (g, w - \bar{w})_{L^2(\Omega)} + \alpha \TV(w)-\alpha \TV(\bar{w})\\
\text{s.t.}\quad & \|w - \bar{w}\|_{L^1(\Omega)} \le \Delta
\text{ and }w(x) \in W \text{ for a.e.\ } x \in D,\\
& w(x) = \bar{w}(x) \text{ for a.e.\ } x \in \Omega\setminus D,
\end{aligned}
\right.
\tag{TRP}
\end{gather}
where we highlight that the integrality is enforced on the subproblem level and
the trust region is imposed in $L^1$, which together with the discretization of the $\TV$-term
leads to a polyhedral constraint structure after discretization.
We briefly note that $\operatorname{\ref{eq:trp}}(w, g, D, \Delta)$
admits a minimizer.
\begin{theorem}[Proposition 3.2 in \cite{leyffer2022sequential}] Let 
	\cref{ass:calD_open_cover,ass:standing}
	hold. Let $D \in \calD$, $w \in \BVW(\Omega)$,
	$g \in  L^2(\Omega)$, and $\Delta > 0$. Then
	$\operatorname{\ref{eq:trp}}(w, g, D, \Delta)$
	admits a minimizer.
\end{theorem}

\section{Benchmark Problems}\label{sec:benchmark_problems}
For the following benchmark problems, we optimize \eqref{eq:p}, where $F(w)$ has the tracking-type form
\begin{equation}
\label{eq:loss}
F(w) = j\circ S(w) \enskip\text{ with }\enskip j(u) = \frac{1}{2}\|u - u_d\|_{L^2(D,\mathbb{K})}^2, 
\end{equation}
$u = S(w)$ is the solution to a corresponding PDE, and $u_d$ is prescribed data. We select benchmark problems for four different PDEs in the
following subsections: \texttt{AD} \cite{baraldi2025domain} features an advection-diffusion PDE with the choices $D = \Omega$, 
$\mathbb{K} = \R$ in $F$; \texttt{Choupi} \cite{schiemann2024integer} is a denoising, deblurring, segmentation problem with the choices
$D = \Omega$, $\mathbb{K} = \R$; \texttt{Exact} \cite{schiemann2024integer} features a linear elliptic PDE with the choices $D = \Omega$,
$\mathbb{K} = \R$ in $F$ with data so that \eqref{eq:p} admits a unique global minimizer; and \texttt{Helmholtz} 
\cite{haslinger2015topology,leyffer2021convergence} features a Helmholtz equation with Robin boundary conditions with the choices
$\mathbb{K} = \mathbb{C}$ and $D = [0.25,1.75] \times [1.6125,2]$. For all of our benchmark problems, the PDE is solved on a
$128\times 128$ grid of squares, which are split into four triangles each, and on which continuous Lagrange ansatz functions of order one
are used for the state variable. The control variable is discretized using a piecewise constant ansatz on the $128\times 128$ square grid,
restricted to the subdomain that the control is acting on for the respective problem.
In order to use the convergent discretization for the $\TV$-term from \cite{schiemann2025discretization}, see our summary
in \cref{sec:w_tvh_discretization} further down below, the control grid is embedded into a coarser grid so that $4\times 4$ square grid cells
constitute a grid cell of the coarser grid.

All computations are performed on identical nodes of TU Dortmund's Linux HPC cluster LiDO3 that have two AMD EPYC 7542 32-Core CPUs and
64 GB RAM (computations were restricted to one CPU). We use DOLFINx 0.9.0 \cite{Dolfinx} for the finite element discretization
of the PDEs and the total variation seminorm and Gurobi 11.0.0 \cite{Gurobi} to solve the generated instances of \eqref{eq:trp}%
\footnote{Code for implementation of PDEs and numerical studies
can be found in \href{https://github.com/paulmanns/ioc-tv-2d-benchmarks}{https://github.com/paulmanns/ioc-tv-2d-benchmarks}.}.

\subsection{\texttt{AD} featuring an advection-diffusion PDE}
This example was used in the numerics of \cite{baraldi2025domain}, which contains the complete description.
We choose $W = \{0,1\}$ control and advection-diffusion equation on the square domain $\Omega = (0,1)^2$.
For a given control $w$ with $w(x) \in  W$ a.e., the solution of the state equation $u\coloneqq S(w)$ is 
determined by the following system of equations
\begin{gather}\label{eq:ad2d}
\begin{aligned}
-\varepsilon \Delta u + c_1 \cdot \nabla u + c_2 u w &= f \quad\text{in } \Omega, \\
u&= 0 \quad\text{on } \{0,1\} \times (0,1) \cup ((0,0.25) \cup (0.75,1)) \times \{0\}, \\
u&= \sin(2 \pi (x_1 - 0.25)) \quad\text{on } (0.25,0.75) \times \{0\}, \\
\partial_n u &= 0 \quad\text{on } (0,1) \times \{1\},
\end{aligned}
\end{gather}
where $c_2 = 2$, $c_1(x) = (\begin{matrix} \sin(\pi x_1)
& \cos(2 \pi x_2)\end{matrix})^T$ for $x \in \Omega$,
$f(x) = \sin(2 \pi x_1 + 2 \pi x_2) + 3$ for $x \in \Omega$,
and $\varepsilon = 4\cdot10^{-2}$.
The state is discretized with first-order Lagrangian finite elements.  
The data $u_d$ is computed by solving a variant of \eqref{eq:ad2d}, where $c_1$
is replaced by $\tilde{c_1}(x) = (\begin{matrix} -x_2
& 2 x_1\end{matrix})^T$, for $w = 2.5 \chi_{A} - 4(x_1 - 0.35)^3 \chi_A - 6(x_2 - 0.35)^3\chi_B$
with $A = (0,0.35)^2$ and $B = \Omega\setminus (0,0.35)^2$.
Regarding the parameters, we scale the $\TV$-term by $\alpha=10^{-3}$.

\subsection{\texttt{Choupi} featuring a linear elliptic PDE}
We choose $W = \{0, \hdots, 5\}$ control on the square domain 
$\Omega = (0,1)^2$. For a given control $w(x) \in W$ a.e., the solution to the
PDE $u \coloneqq S(w)$ is determined by the following linear PDE
\begin{gather}\label{eq:choupi2d}
\begin{aligned}
-\varepsilon \Delta u + u &= w\quad\text{in } \Omega \\
\partial_n u &= 0 \quad\text{on } \partial \Omega,
\end{aligned}
\end{gather}
where the state is again discretized with first-order Lagrangian finite elements. 
We choose $\varepsilon = 10^{-3}$ and the TV regularization term $\alpha=10^{-3}$ from \eqref{eq:p}.
We refer to \S6.4 in \cite{schiemann2024integer} for further details on this problem. The reference state $u_d$
is computed by solving \eqref{eq:choupi2d} for a noisy variant of the open-source grayscale image from \S6.4
in \cite{schiemann2024integer}\footnote{The image can be downloaded from https://github.com/annikaschiemann/choupi.},
which was scaled to the range $[0,1] = \conv W$. The noise is computed by adding Gaussian random noise with zero mean
and $0.05$ standard deviation each pixel and then clamping the resulting pixel value to $[0,1]$.
Regarding the parameters, we choose $\varepsilon = 10^{-3}$ in \eqref{eq:choupi2d} and scale the $\TV$-term by $\alpha=10^{-3}$.

\subsection{\texttt{Exact} featuring a linear elliptic PDE}
\label{subsec:exact}
We choose $W = \{0, 1\}$ control on the square domain 
$\Omega = (0,2)^2$. For a given control $w(x) \in W$ a.e., the solution to the PDE $u \coloneqq S(w)$ 
is determined by the following linear (affine) PDE
\begin{gather}\label{eq:exact2d}
\begin{aligned}
-\varepsilon \Delta u + u &= w + f \quad\text{in } \Omega \\
u &= 0 \quad\text{on } \partial \Omega.
\end{aligned}
\end{gather}
Here, $u_d$ and $f$ are constructed such that the control $\bar{w}$ given by 
\begin{align*}
	\bar w(x) = \begin{cases} 1\ : & x \in B, \\ 0 \ : & x \not\in B, \end{cases} 
\end{align*}
where $B$ is the ball with radius $0.5$ centered at $(1,1)^T$, is the unique solution to
the corresponding instance of \eqref{eq:p}. For this construction, the optimality conditions
to the relaxation, where the binarity constraint is dropped, are used. We refer to 
creating \S4.4 in \cite{schiemann2024integer} for further information.
Regarding the parameters, we choose $\varepsilon = 10^{-2}$ in \eqref{eq:choupi2d} and scale the $\TV$-term by $\alpha=10^{-3}$.

\subsection{\texttt{Helmholtz} featuring a Helmholtz equation with Robin boundary conditions}
We choose $W = \{0, 1, 2\}$ control on the square subset $(0.5,1.5)^2$ of the square domain $\Omega = (0,2)^2$.
For a given control $w(x) \in W$ a.e., The solution for the PDE $u = S(w)$ is determined by following Helmholtz equation
with Robin boundary condition
\begin{gather}\label{eq:helm}
\begin{aligned}
-\Delta u  - c_0^2(1 + q\cdot w)u  & = c_0^2 q\cdot w \cdot u_0 &\quad\text{in }\Omega,\\
\partial u/\partial n - ic_0 u &= 0 &\quad\text{on }\partial\Omega;
\end{aligned}
\end{gather}
see \cite{haslinger2015topology,leyffer2021convergence} for details on this PDE and the resulting Helmholtz cloaking
problem. Regarding the parameters, we choose $c_0 = 6 \pi$, $q = 0.75$, $u_{0} = \cos(c_0y) + i\sin(c_0y)$, and
scale the $\TV$-term by $\alpha=10^{-3}$. In the objective, we choose $u_d = -u_0$.

\section{Performance Improvements for \ref{alg:slip}}\label{sec:performance_improvements_for_slip}
We now detail performance improvements for the fundamental \ref{alg:slip}
algorithm, which is the trust-region algorithm for solving \eqref{eq:p}
without patches or randomness \cite{leyffer2021convergence,schiemann2024integer}. 
We start by describing the algorithm itself, 
which relies on efficient subproblem solves of \eqref{eq:trp}. 
As such, we subsequently describe performance improvements in
solving \eqref{eq:trp}, starting from discretization details
and moving on to coarse/fine grids, cutting plane solvers, 
branching priorities,
lazy constraints, heuristic initializations, 
and early termination. We emphasize that
while these speedups are listed in this section, they subsequently
apply to other variations of \ref{alg:slip} as each requires solving
some form of \eqref{eq:trp}. 

\subsection{SLIP Algorithm Description}\label{sec:slip}
\ref{alg:slip} is a standard trust-region algorithm applied to 
mixed-integer PDE-constrained optimization problems of the form 
\eqref{eq:p}. The most simple variant was proposed and analyzed in 
\cite{leyffer2021convergence,manns2023on}.
In particular, one solves \eqref{eq:trp} for a trust-region radius $\Delta^{\max}2^{-k}$, 
which decreases in the inner $k$-loop. If this solution produces decrease of the linear model $\pred^{n,k}$ and
additionally decreases the full problem $\ared^{n,k}$ by a fraction of $\pred^{n,k}$, then the step is accepted.
Otherwise, one checks if the predicted reduction is negative or zero, 
terminating if true as a stationary point is found. If neither condition holds, then the radius shrinks and
a new solution of \eqref{eq:trp} is found. When a step is accepted,
the trust-region radius is updated as follows. If the current trust-region radius satisfies $\Delta^{\max} 2^{-k} < \Delta^r
 = \Delta^{\max} 2^{-k_r}$, that is,
it is below the reset trust-region radius, the radius is reset to $\Delta^r = \Delta^{\max} 2^{-k_r}$. If it is greater than or equal
to the reset trust-region radius, the trust-region radius
is doubled when the ratio of actual and predicted reduction
is above the ratio threshold $\sigma_u$ and does not exceed
the maximum trust-region radius $\Delta^{\max}$. Otherwise,
the trust-region radius is kept. We highlight that due to this
combination of a typical trust-region update with the reset
at lower trust-region radii, \ref{alg:slip}
differs slightly from \cite{leyffer2022sequential,manns2023on},
where the radius is always reset, making our algorithm faster
in practice. Nevertheless, the convergence theory
of \cite{leyffer2022sequential} for $d = 1$
and \cite{manns2023on} for $d \ge 2$, still applies because the
iterations at larger trust-region radii are not important for the
convergence analysis.
\begin{algorithm}[H]
\def\thealgorithm{SLIP}
\refstepcounter{algorithm}
\caption*{\textbf{Algorithm SLIP: Sequential linear integer programming method}}\label{alg:slip}
\textbf{Input:} $F$ satisfying \cref{ass:standing}, $\Delta^{\max} > 0$,
$k_0$, $k_r \in \N$, $\Delta^r = \Delta^{\max} 2^{-k_r}$,
$w^0 \in \BVW(\Omega)$, $0 < \sigma < \sigma_u < 1$.

\begin{algorithmic}[1]
	\State $k_s \gets k_0$
	\For{$n = 0,1,2,\ldots$}
	\For{$k = k_s,k_s + 1,\ldots$}	
	\State $\tilde{w}^{n,k} \gets$ minimizer of 
	$\text{{\ref{eq:trp}}}(w^{n}, \nabla F(w^{n}), \Omega, \Delta^{\max} 2^{-k})$.\label{ln:slip_trp}
	\State $\pred^{n,k} \gets (\nabla F(w^n), w^n - \tilde{w}^{n,k})_{L^2}
+ \alpha \TV(w^{n}) - \alpha \TV(\tilde{w}^{n,k})$\label{ln:slip_pred} 
	\State $\ared^{n,k} \gets F(w^{n}) + \alpha \TV(w^{n})
- F(\tilde{w}^{n,k}) - \alpha\TV(\tilde{w}^{n,k})$
	\If{$\ared^{n,k} \ge \sigma \pred^{n,k}$ and $\pred^{n,k} > 0$}
	\State $w^n \gets \tilde{w}^{n,k}$
	\If{$k > k_r$}
	\State $k_s \gets k_r$
	\ElsIf{$\ared^{n,k} \ge \sigma_u \pred^{n,k}$}
	\State $k_s \gets \min\{0, k - 1\}$
	\EndIf
	\State \textbf{break inner loop}
	\ElsIf{$\pred^{n,k} \le 0$}
	\State \textbf{return} $w^n$ is stationary.
	\EndIf
	\EndFor
	\EndFor
\end{algorithmic}
\end{algorithm}

\subsection{Discretization Details}\label{sec:discretization}

We assume that our computational domain
$\Omega$ and the patches $D \in \calD$ are two-dimensional, 
axis-aligned rectangles, as is the case for our computational
examples below. We note
that many of the results and considerations presented here
may be carried over to higher dimensions and more complicated
domains; we leave such substantial technical
overhead to future work. First, we summarize our employed discretization in
\cref{sec:w_tvh_discretization,sec:trp_discretization}.

\subsubsection{Discretization of \texorpdfstring{$\TV$}{TV}}\label{sec:w_tvh_discretization}
We use the novel, convergent discretization for
integer optimization  problems with $\TV$-regularization from
\cite{schiemann2025discretization}. 
We assume
that $\Omega$ and the patches are discretized using two meshes
with the following characteristics.
\begin{assumption}[Modification of {\cite[Assumption 2.3]{schiemann2025discretization}}]
We make the following assumptions on the discretization:
\begin{enumerate}
\item The $\TV$-mesh:\quad For $h>0$, $\Omega$ is partitioned
into axis-aligned squares $Q \in \calQ^h$ of height $h$,
$\Omega = \bigcup_{Q \in \calQ^h} Q$. For all $D \in \calD$,
there is a subset $\calQ_D^h \subset \calQ^h$ such that
$D = \bigcup_{Q \in \calQ_D^h}$.
\item The $w$-mesh:\quad For $\tau_h \in (0,h]$, $\Omega$ is 
partitioned into axis-aligned squares $Q \in \calQ^{\tau_h}$ of
height $\tau_h$ that are embedded into the $\TV$-mesh.
Specifically, for each $\tilde{Q} \in \calQ^h$ there exists $\calQ^{\tau_h}_{\tilde{Q}} \subset \calQ^{\tau_h}$ such that
$\bigcup_{Q \in \calQ^{\tau_h}_{\tilde{Q}}}Q = \tilde{Q}$. 
\end{enumerate}
\end{assumption}

Now, the $\TV$-term is discretized using the $\TV$-mesh as
\begin{equation}
\label{eq:tvh}
\TV^h(w)
   \coloneqq \sup\Bigl\{
   \int_{\Omega} w\dvg \phi \dd x : \phi \in RT0_0^h
   \text{ and }
   \|\phi\|_{L^\infty(\Omega,\R^2)} \le 1
   \Bigr\} 
\end{equation}
for $w \in L^2(\Omega)$, where $RT0_0^h \subset H(\dvg,\Omega)$ is
the space of Raviart--Thomas finite-element functions of lowest
order with vanishing normal trace that are defined on the
square (in particular quadrilateral) cells $Q \in \calQ^h$.
The control function $w$ is discretized with the piecewise constant
ansatz on $\calQ^{\tau_h}_D$ on the current patch $D$,
$w|_D \in DG0^{\tau_h}_D$, that is
$w|_D = \sum_{Q \in \cal{Q}^{\tau_h}_D} w_Q \chi_{Q}$
with $w_Q \in [w_0, w_M]$. 
To allow for boundedness
of the sequence of iterates, $\TV^h$ is modified to
$\max\left\{\tfrac{1}{\sqrt{2}}\TV,\TV^h\right\}$ on $DG0^{\tau_h}_D$,
yielding a convergent discretization of $\TV$; see
\cite[Theorems 3.2-3.4]{schiemann2025discretization}.

\subsubsection{Discretization of \texorpdfstring{\eqref{eq:trp}}{TRP}}\label{sec:trp_discretization}
In total, we obtain the
following,  discretized problem: 
\begin{gather*}
\begin{aligned}
\operatorname{TRP}^h(&\bar{w}, g, D, \Delta) \coloneqq\\
&\left\{
\begin{aligned}
\min_{w|_D \in DG0^{\tau_h}_D}\ & (g, w - \bar{w})_{L^2} + \alpha \max\bigl\{\tfrac{1}{\sqrt{2}}\TV(w),\TV^h(w)\bigr\}-\alpha \max\bigl\{\tfrac{1}{\sqrt{2}}\TV(\bar{w}),\TV^h(\bar{w})\bigr\}\\
\text{s.t.}\quad & \|w - \bar{w}\|_{L^1} \le \Delta
\text{ and }w(x) \in W \text{ for a.e.\ } x \in D,\\
& w(x) = \bar{w}(x) \text{ for a.e.\ } x \in \Omega\setminus D.
\end{aligned}
\right.
\end{aligned}
\end{gather*}

One can show that, after a reformulation using
$\R \cup \{\infty\}$-valued functionals
and coupling $\tau_h\in o(h)$ as $h \searrow 0$,
we obtain that the problems ($\operatorname{TRP}^h$) converge
to $\eqref{eq:trp}$ in a $\Gamma$-convergence
sense, implying that limits of solutions to $\operatorname{TRP}^h$
are solutions to $\operatorname{\ref{eq:trp}}$.
Note that this does not directly follow from the
results in \cite[Section 2.2 \& 2.3]{schiemann2025discretization} due to
the absence of the trust-region constraint therein, which
necessitates an additional approximation step in \Cref{alg:oa}.
We prove this $\Gamma$-convergence result in the appendix \cref{sec:trph_trp}.
Following \S4 in \cite{schiemann2025discretization},
one can solve ($\operatorname{TRP}^h$) in a finite number of steps
using the outer approximation \ref{alg:oa}.
\Cref{alg:oa} alternatingly solves
an instance of \eqref{eq:trph_mip} corresponding to
$\operatorname{TRP}^h$, which only features a finite subset of the
infinitely many inequality constraints defining $\TV^h$, and 
an instance of \eqref{eq:tvh_qp} to compute the most-violated 
inequality of $\TV^h$ to then augment and improve \eqref{eq:trph_mip}.
Note that the domain restriction given by $w(x) = \bar w(x)$ for a.a. $x\in \Omega \backslash D$
reduces to the full domain in the abscense of patches $D$, which is utilized in 
subsequent algorithms \ref{alg:patch_slip} and \ref{alg:randomized_patch_slip}. 
\begin{algorithm}[h!]
\caption{Outer approximation algorithm for $\operatorname{TRP}^h$;
see Algorithm 4.1 in \cite{schiemann2025discretization}.}\label{alg:oa}
\textbf{Input:} $\alpha > 0$, $\tau_h > 0$, $h >0$, $\Delta > 0$, 
$\bar{w}\in \BVW(\Omega)$.
\begin{algorithmic}[1]
	\State $k \gets 0$, $\Phi \gets \emptyset$.
	\State $(w^k,T^k) \gets$ optimal solution to
	\begin{gather}\label{eq:trph_mip}
	\begin{aligned}
	\min_{(w,T) \in P0^{\tau_h}_D \times \R} \enskip & (g,w)_{L^2(D)} + \alpha T \\
	\mathrm{s.t.} \enskip\, & \TV(w) \leq cT \\
	& \max_{\phi \in \Phi}\int_\Omega w \dvg \phi\,\dd x \leq T\\
	& \|w - \bar{w}\|_{L^1(D)} \le \Delta \\
	& w(x) \in \{w_1,\ldots,w_M\}
	&&\text{for all } x \in D \\
	&  w(x) = \bar{w}(x)
	&&\text{for a.a.\ } x \in \Omega\setminus D.	
	\end{aligned}\tag{$\operatorname{TRP}^h$-$\Phi$}
	\end{gather}
	\State $\phi^{k+1} \gets$ optimal solution to
	\begin{gather}\label{eq:tvh_qp}
	\begin{aligned}
	\max_{\phi}\enskip & \int_{\Omega} \dvg \phi (x) \, w^k(x) \dd x \quad
	\mathrm{s.t.} \quad \phi \in RT0^h_0, \enskip \| \phi \|_{L^\infty(\Omega,\R^2)} \leq 1.
	\end{aligned} \tag{$\TV^h$-$\phi$}
	\end{gather}
	\State $\Phi \gets \Phi \cup \{\phi^{k+1}\}$.
	\If{$\int_\Omega \dvg \phi^{k+1}(x) \, w^k(x) \dd x - T^k \leq 0$}
	\State \Return $w^k$ as optimal solution.
	\EndIf
	\State $k \gets k+1$ and go to Step 2.
\end{algorithmic}
\end{algorithm}

The uniform discretization into square cells directly provides a corresponding underlying graph structure in the form of an 
$N_x \times N_y$ grid graph; see \cite{manns2025discrete}.  The nodes of the graph correspond to the entries $d_{i,j}$ while the
edges model the absolute value terms in the objective. The edges can be divided into row and column edges. 
Let $W_{\min} = \min\{W\}$ and likewise for the maximum.  The discretized problem \eqref{eq:trph_mip} can be written
in the manner of \cite{manns2025discrete}
\begin{gather}
\begin{aligned}
\min_{d,\delta,\beta,\gamma}\quad   &\sum_{i=1}^{N_x} \sum_{j=1}^{N_y} c_{i,j} d_{i,j} + \alpha  \left (\sum_{i=1}^{N_x-1} \sum_{j=1}^{N_y}  \beta_{i,j} +  \sum_{i=1}^{N_x} \sum_{j=1}^{N_y-1} \gamma_{i,j} \right )\\
\text{s.t.}\quad
&W_{\min} \leq \bar w_{i,j} + d_{i,j} \leq W_{\max} 
&&\text{ for all } i \in [N_x], j \in [N_y],\\
&-\beta_{i,j} \leq \bar w_{i+1,j} + d_{i+1,j} - \bar w_{i,j} - d_{i,j} \leq \beta_{i,j} 
&&\text{ for all } i \in [N_x-1], j \in [N_y],\\
&-\gamma_{i,j} \leq \bar w_{i,j+1} + d_{i,j+1} - \bar w_{i,j} - d_{i,j} \leq \gamma_{i,j}
&&\text{ for all } i \in [N_x], j \in [N_y-1],\\
&-\delta_{i,j} \leq d_{i,j} \leq \delta_{i,j}
&&\text{ for all } i \in [N_x], j \in [N_y], \\
&\sum_{i=1}^{N_x} \sum_{j=1}^{N_y} \delta_{i,j} \leq \Delta,\quad
d \in \mathbb{Z}^{N_x \times N_y},
\end{aligned}
\tag{IP}
\label{eq:ip}
\end{gather}
where we introduce auxiliary variables 
$\beta$,  $\gamma$ and $\delta$ as a linear integer program. 
The variable $\beta$ models the jumps along the rows of the 
grid graph while $\gamma$ models jumps along the columns of the graph.  
The summation over delta, $\sum_{i=1}^{N_x} \sum_{j=1}^{N_y} \delta_{i,j} \leq \Delta$,
is effectively a 
capacity constraint bounded by the trust-region radius. 
We note that the general formulation of \eqref{eq:trph_mip} with patches
has the additional constraint
$w(x) = \bar w(x) \text{ for } x \in \Omega \setminus D$, 
which is not given in \eqref{eq:ip}. 
Inclusion of patches warrants 
additional constraints of the form $d_{i,j} = 0$ for the corresponding out-of-domain 
entries. 
These constraints however do not affect the following arguments and are thus dropped 
here for sake of clarity and comparison to \cite{manns2025discrete}. 
These constraints also be enforced via setting the coefficients for the variables $\delta_{i,j}$ sufficiently large instead of $0$.
\cite[Theorem 3]{manns2025discrete} it was shown that every vertex solution for the 
linear programming relaxation to \eqref{eq:ip} has the special property 
that fractional values only occur in entries of $d$ such that the corresponding nodes 
in combination with the connecting edges form a connected subgraphs.  
This connected subgraphs was referred to as the fractional component. 
Furthermore, the values of $\bar w + d$ are the same for all entries in the fractional component.

\subsection{Performance Improvements of \texorpdfstring{\eqref{eq:trph_mip}}{TRP-h}}
The solution of \eqref{eq:ip} (and hence \eqref{eq:trph_mip}) is strongly NP-hard under 
the assumption that the bisection problem is also NP-hard on grid graphs with a finite
number of holes \cite{manns2025discrete}. The latter has been conjectured to be NP-hard
in \cite{papadimitriou1996bisection}. This motivates us to solve \eqref{eq:trph_mip}
using a branch-and-bound-based solver, which we accelerate using the integer
programming techniques that are summarized below.

\subsubsection{Cutting Planes for \eqref{eq:ip}}
In \cite{manns2025discrete}, it was demonstrated that the inclusion of a class of cutting planes
can reduce the run time significantly. These cutting planes take advantage of the fact that
the solution to a continuous relaxation of \eqref{eq:ip} has at most one connected component
of nodes in the graph, where a non-integer value is attained for $d_{i,j}$ and this value is
the same for the whole component; see also \cite{yang2025specialized}. The cutting planes
then are based on the minimum cut ratio for these fractional components of relaxation solutions
to relate the amount of capacity spent (that is, trust-region consumption) inside the component
to additional jumps that must occur and increase the total variation cost in integer feasible
solutions. This has been derived in detail in \cite{manns2025discrete} for the case $W = \{0,1\}$
and an extension to the non-binary case, which we require for our test cases \texttt{Helmholtz}
and \texttt{Choupi} is given in \S8.3.3 in \cite{severitt2025integer}.
We highlight that \eqref{eq:ip} has to be solved at the start of \cref{alg:oa}, the generated cutting
planes stay valid during the addition of the additional constraints due to the total variation
discretization, and may be generated if they exist when such constraints have been added already.

\subsubsection{Early termination}
%
% A key insight in trust-region algorithm analysis is that it 
In general, it is not necessary to solve trust-region
subproblems to global optimality to obtain convergence to stationary points. One only has to
compute a feasible point that realizes sufficient decrease in the model. 
In smooth, unconstrained optimization,
a typical sufficient decrease condition 
% is to realize at least as much decrease as a 
% steepest descent step
% along the model inside the trust region, called 
% This point of comparison, which can usually be computed inexpensively,
some fraction of the Cauchy point \cite[\S6]{conn2000trust}, 
which is often the steepest descent step. 
% A related
% argument is still possible in our setting. 
Assuming that we can solve our current instance of \eqref{eq:trp} up
to a constant factor of suboptimality, e.g., our computed solutions $\tilde{w}^{n,k,D_n}$
% that realize the predicted reduction $\pred^{n,k,D_n}$ do not necessarily minimize
% $\operatorname{\ref{eq:trp}}(w^{n}, \nabla F(w^{n}), D_n, \Delta_{0} 2^{-k})$ but instead satisfy
%
satisfies
\begin{gather}\label{eq:fixed_approximation_factor}
\pred^{n,k,D_n} \ge -\rho \min\operatorname{\ref{eq:trp}}(w^{n}, \nabla F(w^{n}), D_n, \Delta_{0} 2^{-k})
\end{gather}
for some fixed $\rho \in (0,1)$. 
Then, this factor occurs linearly in the estimates of the predicted
reduction and the remainder terms that estimate the mismatch between $F$ and the model approximation.
In particular, the asymptotic behavior of all terms remains the same,
and hence \cite[Lemma 5.3, Theorem 5.4, \& Theorem 5.8]{baraldi2025domain} 
and
\cref{thm:main_convergence_theorem}, i.e. the convergence
results for \ref{alg:patch_slip} and \ref{alg:randomized_patch_slip}, 
still hold.

When solving a discretized instance $\operatorname{TRP}^h$ of \eqref{eq:trp} to global optimality
with a branch-and-bound algorithm, the algorithm generates a sequence of incumbent solutions
with corresponding objective values, where the best achieved objective value is the
primal (in our case upper) bound $U$. 
The primal bound $U$ is compared to dual (in our case lower)
bounds that arise from by solving continuous relaxations with partial fixations generated by
branching. 
The algorithm terminates when the relative difference between $U$ and the smallest
of the current dual bounds $L$ falls below a given threshold $\varepsilon_{\textrm{gap}}$,
the so-called relative duality gap.
Thus, since $U$ is the negative of the predicted reduction for the incumbent solution, a
sufficient condition for the discretized condition $\pred^{n,k,D_n} \ge - \rho \min \operatorname{TRP}^h$
of \eqref{eq:fixed_approximation_factor} to hold is to terminate early once
$U \le \rho L$ and $L \le 0$ hold. Note that the optimal objective value of \eqref{eq:trp}
is always non-positive since the linearization point itself is feasible with objective
value zero.

In addition, the trust-region algorithm asymptotics are always guaranteed by the behavior for small
values of the trust-region radius due to the model approximation properties. It is thus possible to
accept suboptimal steps for large trust-region radii as long as they yield just some 
descent in the objective.

We thus propose the following early termination criterion for the subproblem solver.
If a minimum time $t_1$ has elapsed, a distinction is made depending on the
condition $\Delta \le \Delta_{\textrm{small}}$. If this is evaluated to true,
the solution process is terminated if $U \le \rho L$ and $\rho L \le -\varepsilon_{\textrm{num}}$
hold. Otherwise, it is terminated if $U \le \rho L$ holds and $\rho L \le -\varepsilon_{\textrm{num}}$
or, alternatively, if $U \le -\varepsilon_{\textrm{num}}$ and a maximum time $t_2$ has elapsed.
The value of $\Delta_{\textrm{small}}$ depends on the smallest volume of the grid cells
$Q_i$, $i \in \{1,\ldots,N_h\}$, that partition our domain $\Omega$ and constitute the control
function ansatz $w = \sum_{i=1}^{N_h} Q_i$ in our discretization. The time $t_1$ is there to avoid
preempting fast and powerful presolving techniques, primal heuristics, and cutting plane generations
that are already implemented in the solver. We have tabulated the chosen parameter values
for the early termination criterion in our implementation
in \Cref{tbl:parameter_values_early_termination} below.
\begin{table}
\caption{Parameter values for the early termination criterion.}\label{tbl:parameter_values_early_termination}
\begin{center}
\begin{tabular}{cccccc}
	\toprule
	$\varepsilon_{\textrm{gap}}$ 
	& $\varepsilon_{\textrm{num}}$ 
	& $\rho$ 
	& $t_1$ [s] 
	& $t_2$ [s] 
	& $\Delta_{\textrm{small}}$ \\
	\midrule
	$10^{-4}$ & $10^{-7}$ & $0.5$ & $10$ & $1200$ &  $8 \min\{ |Q_i| : i \in \{1,\ldots,N_h\}$ \\
	\bottomrule
\end{tabular}
\end{center}
\end{table}

\subsection{Lazy constraint handling for the solution of
	\texorpdfstring{($\operatorname{TRP}^h$)}{TRP-h}}
The discretized patch-based trust-region subproblem
($\operatorname{TRP}^h$) and \eqref{eq:trph_mip} are equivalent
if $\Phi$ contains all admissible $RT0^h_0$ vector fields, that
is, if
$\Phi = \{ \phi \in RT0^h_0 : \|\phi\|_{L^\infty(\Omega,\R^2)} \le 1\}$.
In addition, only finitely many elements are required since
the discreteness of $w$ and the piecewise constant ansatz for
$w$ on $\calQ^{\tau_h}$ imply that there are only finitely
many realizations of $w$ for fixed $h$, $\tau_h$.
Consequently, \eqref{eq:trph_mip} becomes an integer linear program
when $(g,w)_{L^2(D)}$ and $\|w - \bar{w}\|_{L^1(D)}$ are replaced by
appropriate quadratures or exact integral evaluation and slack 
variables to linearize the absolute value.
Since integer linear program formulations often feature an
exponential number of inequalities,
general purpose solvers for integer linear programs generally allow
to optimize with a subset of the inequalities and add new inequalities
on the fly. This is done through a callback mechanism that provides
a means to add one ore more violated inequalities once the solver 
discovers a new incumbent candidate on the current subset
of the description. We propose to use this mechanism with
\eqref{eq:tvh_qp} to compute the most-violated inequality once
a new incumbent candidate becomes available in order to avoid
re-solving the problem \eqref{eq:trph_mip} all the time.

\subsection{Heuristic initialization of the trust-region algorithm}\label{sec:heuristic_initialization}
In \cite[\S6.4]{severitt2023efficient},
the number of iterations of SLIP and
its overall runtime consumption were significantly reduced for the case $d = 1$ by
solving a continuous relaxation and 
then rounding the computed solution to an
integer-valued one. 
In the following experiments, we first solve
a smoothed discretization the continuous relaxation of \eqref{eq:p}
given by
\begin{gather}\label{eq:r}
\begin{aligned}
\min_{w \in L^1(\Omega)}\ 
& F(w) + \alpha \TV_\varepsilon(w) \\
\text{s.t.}\quad 
& w(x) \in \conv W = [w_1,w_M]
\text{ for almost every (a.e.) }
x \in \Omega
\end{aligned}\tag{R}
\end{gather}
to approximate global optimality (if \eqref{eq:r} is convex) or approximate stationarity
(if \eqref{eq:r} is nonconvex) with a gradient-based nonlinear programming solver.
Here, we use the same discretizations for $w$ and $F$ as for \eqref{eq:p}
and approximate $\TV$ by a differentiable variant $\TV_{\varepsilon}$.
Then, we round the resulting control, a piecewise constant function, 
to the nearest element in $W$ in every element of the piecewise constant ansatz, where ties
are broken arbitrarily. We use SciPy's \cite{2020SciPy-NMeth} implementation of
the L-BFGS-B algorithm from \cite{liu1989limited} as nonlinear programming solver, where
we specified a maximum iteration number of $1\,000$
and the projected gradient tolerance as $10^{-8}$.

Regarding $\TV_{\varepsilon}$, we first note that for a piecewise constant control function
ansatz $w = \sum_{i=1}^{N_h} w_i \chi_{Q_i}$ for a uniform decomposition of
the domain $\Omega$ into squares $Q_1$, $\ldots$, $Q_{N}$ for some $N \in \N$.
the derivative can be approximated in an isotropy-conforming way by means of a centralized
finite-difference stencil proposed in \cite{lai2009convergence}; see also \S1 in
\cite{caillaud2023error}, which results in the formula
\[
\TV_{\textrm{FD}}(w) \coloneqq h^2 \sum_{i=1}^{N_h} \bigl\|
\bigl(\begin{matrix}
(D_x w)_{i}
& (D_y w)_{i}
\end{matrix}
\bigr)^T
\bigr\|_2,
\]
where $(D_x w)_i$ denotes the finite-difference stencil applied to $w$
evaluated for cell $i$ and $h$ is the side length of the cells $Q_i$.
Since this is nonsmooth due to the singularity in the derivative of
the square root in the Euclidean norm, we replace 
$\TV_{\textrm{FD}}(w)$ by
\[ \TV_{\varepsilon}(w)
\coloneqq
h \sum_{i=1}^N \Phi_\varepsilon\Bigl(
h\bigl(\begin{matrix}
(D_x w)_{i}
& (D_y w)_{i}
\end{matrix}
\bigr)^T\Bigr), 
\]
where $\Phi_\varepsilon$ is the Huber loss, that is,
\[ \Phi_\varepsilon(v)
\coloneqq
\left\{
\begin{matrix}
\|v\|^2_2/(2\varepsilon) & \text{ if } \|v\| \le \varepsilon \\
\|v\| - \varepsilon/2 & \text{ else.}
\end{matrix}
\right.
\]
The Huber loss is used in combination with a different difference stencil
in \cite{chambolle2017accelerated}; see in particular the appendix therein. 
In our experiments, we use $\varepsilon = 10^{-2}$.

\subsection{Numerical Assessment}
We run \ref{alg:slip} on our four benchmark problems from 
\cref{sec:benchmark_problems} with the intention
to check if the proposed improvements / their combination
reduces the overall compute time of \ref{alg:slip}.
We report the objective values up to (only) three digits of
precision because we do not want to put too much weight on
very small differences in performance.

\subsubsection{Test Configuration 1}
In our first test configuration, we always initialize \ref{alg:slip}
with $w^0 \equiv 0$ and test the effect of the lazy constraint mechanism
(\texttt{Lazy = T/F}) for the handling of $\TVh$ from \eqref{eq:tvh}, the 
cutting plane strategy (\texttt{Cuts = T/F}) from \eqref{eq:ip}, and the 
early termination criterion (\texttt{Early = T/F}). We prescribe a limit 
of $20$ cumulative inner iterations and a time limit of $72$ hours. This 
makes eight settings in total that are tested against each other for all 
of the four test problems. 

If the early termination is activated (\texttt{Early = T}),
\ref{alg:slip} terminates after hitting the limit of cumulative
inner iterations within two hours, in many cases after a few minutes.
For all benchmark problems except \texttt{HELMHOLTZ}, where the
iteration limit is always reached, \ref{alg:slip}
terminates after reaching time limit of $72$ hours when
\texttt{Early = F}. Except for \texttt{HELMHOLTZ}, the achieved
objective values for the setting \texttt{Early = T} are lower than
using \texttt{Early = F}.
For \texttt{HELMHOLTZ}, the results are the opposite, which is
unsurprising since the number of accepted steps is between 9 and
11 for all eight settings and solving \eqref{eq:trp}
to global optimality if possible 
is likely to produce a step of higher quality. The compute times
are much lower, approximately between half and one order of magnitude
in favor of \texttt{Early = T} for \texttt{HELMHOLTZ}.

Overall, we can conclude that \texttt{Early = T} is a beneficial
option. For the lazy constraint and cutting plane mechanism, the
results are inconclusive and also difficult to assess due to the low 
compute times when the cumulative inner iteration limit is reached, which 
we therefore increase to $1000$ for the next test configuration.
We report the running times, cumulative inner iterations, accepted
steps,  achieved objective value for this test configuration
in \cref{tbl:slip_test1}.
\begin{table}[h]
\caption{Running times, cumulative inner iterations, accepted
steps, achieved objective value for Test Configuration 1
to assess the improvements of \ref{alg:slip} from
\cref{sec:slip}. Winners in terms of objective value and
compute times within the reported precision are highlighted with pink color.}\label{tbl:slip_test1}
\begin{adjustbox}{width=\textwidth}	
\begin{tabular}{l|llllllll}
\toprule
& \texttt{Lazy = T} 
& \texttt{Lazy = T} 
& \texttt{Lazy = T} 
& \texttt{Lazy = T} 
& \texttt{Lazy = F}
& \texttt{Lazy = F}
& \texttt{Lazy = F}
& \texttt{Lazy = F}
\\
& \texttt{Cuts = T}
& \texttt{Cuts = T}
& \texttt{Cuts = F}
& \texttt{Cuts = F}
& \texttt{Cuts = T}
& \texttt{Cuts = T}
& \texttt{Cuts = F}
& \texttt{Cuts = F} \\
& \texttt{Early = T}
& \texttt{Early = F} 
& \texttt{Early = T}
& \texttt{Early = F}
& \texttt{Early = T}
& \texttt{Early = F} 
& \texttt{Early = T}
& \texttt{Early = F} \\
\midrule
&\multicolumn{8}{c}{\texttt{AD}} \\
Running time [h]
& \adjustbox{bgcolor=pink}{$1.19$}                 % 111
& $72$ (limit)           % 110   
& $1.24$                 % 101
& $72$ (limit)           % 100
& $1.92$                 % 011
& $72$ (limit)           % 010   
& $1.92$                 % 001
& $72$ (limit)           % 000
\\
Cum.\ inner iterations 
& 20 (limit)             % 111
& 3                      % 110   
& 20 (limit)             % 101
& 3                      % 100
& 20 (limit)             % 011
& 2                      % 010   
& 20 (limit)             % 001
& 2                      % 000
\\
Accepted steps 
& 10             % 111
& 3              % 110   
& 10             % 101
& 3              % 100
& 11             % 011
& 2              % 010   
& 10             % 001
& 2              % 000
\\
Objective achieved
& \adjustbox{bgcolor=pink}{$6.74 \times 10^{-1}$}  % 111
& $6.75 \times 10^{-1}$  % 110   
& \adjustbox{bgcolor=pink}{$6.74 \times 10^{-1}$}  % 101
& $6.75 \times 10^{-1}$  % 100
& \adjustbox{bgcolor=pink}{$6.74 \times 10^{-1}$}  % 011
& $6.93 \times 10^{-1}$  % 010   
& \adjustbox{bgcolor=pink}{$6.74 \times 10^{-1}$}  % 001
& $6.93 \times 10^{-1}$  % 000
\\
\midrule
&\multicolumn{8}{c}{\texttt{CHOUPI}} \\
Running time [h]
& $0.0836$               % 111
& $72$ (limit)           % 110   
& $0.0758$               % 101
& $72$ (limit)           % 100
& $0.0694$               % 011
& $72$ (limit)           % 010   
& \adjustbox{bgcolor=pink}{$0.0638$}               % 001
& $72$ (limit)           % 000
\\
Cum.\ inner iterations 
& 20 (limit)             % 111
& 1                      % 110   
& 20 (limit)             % 101
& 1                      % 100
& 20 (limit)             % 011
& 1                      % 010   
& 20 (limit)             % 001
& 1                      % 000
\\
Accepted steps 
& 11             % 111
& 1              % 110   
& 11             % 101
& 1              % 100
& 11             % 011
& 1              % 010   
& 11             % 001
& 1              % 000
\\
Objective achieved
& \adjustbox{bgcolor=pink}{$1.49 \times 10^{-2}$}  % 111
& $2.12 \times 10^{-1}$  % 110   
& \adjustbox{bgcolor=pink}{$1.49 \times 10^{-2}$}  % 101
& $2.12 \times 10^{-1}$  % 100
& $1.50 \times 10^{-2}$  % 011
& $2.12 \times 10^{-1}$  % 010   
& \adjustbox{bgcolor=pink}{$1.49 \times 10^{-2}$}  % 001
& $2.12 \times 10^{-1}$  % 000
\\
\midrule
&\multicolumn{8}{c}{\texttt{EXACT}} \\
Running time [h]
& $0.0853$               % 111
& $72$ (limit)           % 110   
& \adjustbox{bgcolor=pink}{$0.0797$} % 101
& $72$ (limit)           % 100
& $0.116$                % 011
& $72$ (limit)           % 010   
& $0.0944$               % 001
& $72$ (limit)           % 000
\\
Cum.\ inner iterations 
& 20 (limit)             % 111
& 1                      % 110   
& 20 (limit)             % 101
& 1                      % 100
& 20 (limit)             % 011
& 1                      % 010   
& 20 (limit)             % 001
& 1                      % 000
\\
Accepted steps 
& 8             % 111
& 1              % 110   
& 8               % 101
& 1              % 100
& 7             % 011
& 1              % 010   
& 7             % 001
& 1              % 000
\\
Objective achieved
& \adjustbox{bgcolor=pink}{$3.21 \times 10^{-3}$}  % 111
& $3.82 \times 10^{-2}$  % 110   
& $3.32 \times 10^{-3}$  % 101
& $3.82 \times 10^{-2}$  % 100
& $3.40 \times 10^{-3}$  % 011
& $3.82 \times 10^{-2}$  % 010   
& $3.40 \times 10^{-3}$  % 001
& $3.82 \times 10^{-2}$  % 000 
\\
\midrule
&\multicolumn{8}{c}{\texttt{HELMHOLTZ}} \\
Running time [h]
& $0.163$                % 111
& $1.82$                 % 110   
& $0.157$                % 101
& $1.19$                 % 100
& $0.151$                % 011
& $0.566$                % 010   
& \adjustbox{bgcolor=pink}{$0.145$} % 001
& $0.558$       % 000
\\
Cum.\ inner iterations 
& 20 (limit)             % 111
& 20 (limit)             % 110   
& 20 (limit)             % 101
& 20 (limit)             % 100
& 20 (limit)             % 011
& 20 (limit)             % 010   
& 20 (limit)             % 001
& 20 (limit)             % 000
\\
Accepted steps 
& 11             % 111
& 9              % 110   
& 11               % 101
& 10              % 100
& 10             % 011
& 10              % 010   
& 10             % 001
& 10              % 000
\\
Objective achieved
& $5.77 \times 10^{-2}$  % 111
& $5.35 \times 10^{-2}$  % 110   
& $5.80 \times 10^{-2}$  % 101
& \adjustbox{bgcolor=pink}{$5.02 \times 10^{-2}$}  % 100
& $5.50 \times 10^{-2}$  % 011
& $5.23 \times 10^{-2}$  % 010   
& $5.50 \times 10^{-2}$  % 001
& $5.23 \times 10^{-2}$  % 000
\\
\bottomrule
\end{tabular}
\end{adjustbox}
\end{table}

\subsubsection{Test Configuration 2}\label{sec:slip_test_config_2}
In our second test configuration, we set activate early
termination for all computations. Again, we test the
effect of the lazy constraint mechanism (\texttt{Lazy = T/F})
and the cutting plane strategy (\texttt{Cuts = T/F}). In addition,
we test two different initialization options for \ref{alg:slip},
$w^0 \equiv 0$ (\texttt{Heur.\ = F}) as before and $w^0$ being 
the solution, that is, approximate stationary point, to a 
continuous relaxation of  \eqref{eq:p} as described in 
\cref{sec:performance_improvements_for_slip}
(\texttt{Heur.\ = T}). Therefore, we again have eight different
settings that are tested on the four benchmark problems. We prescribe a
limit of $1000$ cumulative inner iterations and 
and again a $72$ hour time limit on \ref{alg:slip}. The computation of 
the initial iterate is included if \texttt{Heur.\ = T} to enable
a fair comparison of the achieved objective values.

We observe that all settings reach time limit of $72$ hours for
benchmarks \texttt{AD}, \texttt{CHOUPI}, and \texttt{EXACT}.
For \texttt{HELMHOLTZ}, \ref{alg:slip} terminates for
the four settings with \texttt{Heur.\ = T} due to a contraction
of the trust-region radius, thereby constituting a successful
completion for the fixed discretization. The four settings with 
\texttt{Heur.\ = F} reach the cumulative inner 
iteration limit of $1000$.
For all of the four benchmark problems, the best objective
values within a precision of three digits are achieved (not
always exclusively) for the setting \texttt{Lazy = T},
\texttt{Cuts = T}, and \texttt{Heur.\ = T}.
For \texttt{HELMHOLTZ}, \ref{alg:slip} terminates for
all of the three optimizations producing the best objective value 
without reaching time or cumulative inner iteration limit.
Considering the compute times
for these optimizations, the setting \texttt{Lazy = T},
\texttt{Cuts = T}, \texttt{Heur.\ = T} had the lowest time consumption
of $21.66$ hours compared to $35.42$ hours for
\texttt{Lazy = F}, \texttt{Cuts = T}, \texttt{Heur.\ = T}
and $41.25$ hours for
\texttt{Lazy = F}, \texttt{Cuts = F}, \texttt{Heur.\ = T}.
We therefore conclude that using all of the three proposed
improvements is beneficial and improves the algorithm.

We highlight that \texttt{Heur.\ = T} has by far the largest 
improvement effect and for benchmark problems \texttt{CHOUPI}, 
\texttt{EXACT}, and \texttt{HELMHOLTZ}, all of the optimizations
that produce the best objective within three digits of precision
have \texttt{Heur.\ = T}. For benchmark problems \texttt{AD}, all of
the eight optimizations produce the same objective value within three 
digits of precision within the prescribed time limit.
We believe that this makes sense since the relax-\&-round strategy
essentially bypasses many costly integer optimization problems
and leads to an initialization closer to stationary point.
This is of course most pronounced for \texttt{EXACT}, where
the continuous relaxation is a strictly convex problem that
has a unique binary-valued solution before discretization.
Consequently, we only have zero or one accepted step in \ref{alg:slip}
and almost immediately, a computationally very expensive instance
of \eqref{eq:trp} is generated that cannot be solved within the 
prescribed time limit for \ref{alg:slip}. This often happens when the 
fractional component of the LP relaxation coincides with the whole domain 
or the whole patch so that the LP relaxation is extremely weak, thereby 
causing a very long solution time while at the same time rendering the 
cutting planes inefficient; see also \cite{severitt2025integer}. From our experience, 
this case often happens when the solution to the subproblem is zero, 
causing the algorithm to terminate afterwards.
As a consequence, the results for \texttt{Heur.\ = T} are not very
informative for \texttt{EXACT} and in contrast to the other three
benchmark problems, we will use \texttt{Heur.\ = F} for \texttt{EXACT}.

\begin{table}[h]
	\caption{Running times (including heuristic initalization
		computation if \texttt{Heur.\ = T}), cumulative inner iterations, 
		accepted steps, objective of initial iterate, and achieved objective 
		values for Test Configuration 2 to assess the improvements of 
		\ref{alg:slip} from \cref{sec:slip}. Winners in terms of objective
		value within the reported precision are highlighted with pink color.}
	\begin{adjustbox}{width=\textwidth}	
		\begin{tabular}{l|llllllll}
			\toprule
			& \texttt{Lazy = T}
			& \texttt{Lazy = T}
			& \texttt{Lazy = T}
			& \texttt{Lazy = T}
			& \texttt{Lazy = F}
			& \texttt{Lazy = F}
			& \texttt{Lazy = F}
			& \texttt{Lazy = F}
			\\
			& \texttt{Cuts = T}
			& \texttt{Cuts = T}
			& \texttt{Cuts = F}
			& \texttt{Cuts = F}
			& \texttt{Cuts = T}
			& \texttt{Cuts = T}
			& \texttt{Cuts = F}
			& \texttt{Cuts = F} \\
			& \texttt{Heur.\ = T}
			& \texttt{Heur.\ = F} 
			& \texttt{Heur.\ = T}
			& \texttt{Heur.\ = F}
			& \texttt{Heur.\ = T}
			& \texttt{Heur.\ = F}
			& \texttt{Heur.\ = T}
			& \texttt{Heur.\ = F} \\
			\midrule
			&\multicolumn{8}{c}{\texttt{AD}} \\
			Running time [all,h]
			& 72 (limit)  % 111
			& 72 (limit) % 110   
			& 72 (limit) % 101
			& 72 (limit) % 100
			& 72 (limit) % 011
			& 72 (limit) % 010   
			& 72 (limit) % 001
			& 72 (limit) % 000
			\\
			Running time [heur.,h] 
			& 0.85 % 111
			& 0  % 110   
			& 0.86  % 101
			& 0  % 100
			& 0.86  % 011 
			& 0  % 010   
			& 0.86  % 001
			& 0  % 000
			\\				
			Cum.\ inner iterations 
			& 59 % 111
			& 53  % 110   
			& 36  % 101
			& 93  % 100
			& 11  % 011
			& 29  % 010   
			& 9   % 001
			& 24  % 000
			\\
			Accepted steps 
			& 26 % 111
			& 24 % 110   
			& 14 % 101
			& 46 % 100
			& 3  % 011
			& 17 % 010   
			& 3  % 001
			& 11 % 000			
			\\			
			Initial objective
			& $6.74 \times 10^{-1}$ % 111
			& $1.06$  % 110   
			& $6.74 \times 10^{-1}$ % 101
			& $1.06$                % 100
			& $6.74 \times 10^{-1}$ % 011
			& $1.06$                % 010   
			& $6.74 \times 10^{-1}$ % 001
			& $1.06$                % 000
			\\				
			Objective achieved
			& \adjustbox{bgcolor=pink}{$6.74 \times 10^{-1}$} % 111
			& \adjustbox{bgcolor=pink}{$6.74 \times 10^{-1}$} % 110   
			& \adjustbox{bgcolor=pink}{$6.74 \times 10^{-1}$} % 101
			& \adjustbox{bgcolor=pink}{$6.74 \times 10^{-1}$} % 100
			& \adjustbox{bgcolor=pink}{$6.74 \times 10^{-1}$} % 011
			& \adjustbox{bgcolor=pink}{$6.74 \times 10^{-1}$} % 010   
			& \adjustbox{bgcolor=pink}{$6.74 \times 10^{-1}$} % 001
			& \adjustbox{bgcolor=pink}{$6.74 \times 10^{-1}$} % 000
			\\			
			\midrule
			&\multicolumn{8}{c}{\texttt{CHOUPI}} \\
			Running time [all,h] 
			& 72 (limit) % 111
			& 72 (limit) % 110   
			& 72 (limit) % 101
			& 72 (limit) % 100
			& 72 (limit) % 011
			& 72 (limit) % 010   
			& 72 (limit) % 001
			& 72 (limit) % 000
			\\
			Running time [heur.,s] 
			& 31   % 111
			& 0    % 110   
			& 31   % 101
			& 0    % 100
			& 31   % 011 
			& 0    % 010   
			& 31   % 001
			& 0    % 000
			\\			
			Cum.\ inner iterations 
			& 236  % 111
			& 403  % 110   
			& 210  % 101
			& 400  % 100
			& 104  % 011
			& 295  % 010   
			& 90   % 001
			& 280  % 000
			\\
			Accepted steps 
			& 90  % 111
			& 191 % 110   
			& 82  % 101
			& 192 % 100
			& 42  % 011
			& 146 % 010   
			& 39  % 001
			& 139 % 000			
			\\
			Initial objective 
			& $3.44 \times 10^{-3}$  % 111
			& $2.83 \times 10^{-1}$  % 110   
			& $3.44 \times 10^{-3}$  % 101
			& $2.83 \times 10^{-1}$  % 100
			& $3.44 \times 10^{-3}$  % 011
			& $2.83 \times 10^{-1}$  % 010   
			& $3.44 \times 10^{-3}$  % 001
			& $2.83 \times 10^{-1}$  % 000			
			\\			
			Objective achieved
			& \adjustbox{bgcolor=pink}{$3.05 \times 10^{-3}$}  % 111
			& $3.18 \times 10^{-3}$  % 110   
			& $3.07 \times 10^{-3}$  % 101
			& $3.11 \times 10^{-3}$  % 100
			& $3.16 \times 10^{-3}$  % 011
			& $3.32 \times 10^{-3}$  % 010   
			& $3.15 \times 10^{-3}$ % 001
			& $3.40 \times 10^{-3}$ % 000
			\\
			\midrule			
			&\multicolumn{8}{c}{\texttt{EXACT}} \\
			Running time [all,h]
			&  72 (limit) % 111
			&  72 (limit) % 110   
			&  72 (limit) % 101
			&  72 (limit) % 100
			&  72 (limit) % 011
			&  72 (limit) % 010   
			&  72 (limit) % 001
			&  72 (limit) % 000
			\\
			Running time [heur.,s] 
			&  165  % 111
			&  0    % 110   
			&  165  % 101
			&  0    % 100
			&  164  % 011
			&  0    % 010   
			&  166  % 001
			&  0    % 000
			\\			
			Cum.\ inner iterations 
			& 10  % 111
			& 38  % 110   
			& 10  % 101
			& 42  % 100
			& 0   % 011
			& 60  % 010   
			& 0   % 001
			& 69  % 000
			\\
			Accepted steps 
			& 1  % 111
			& 17 % 110   
			& 1  % 101
			& 18 % 100
			& 0  % 011
			& 28 % 010   
			& 0  % 001
			& 31 % 000			
			\\			
			Initial objective
			& $3.10 \times 10^{-3}$  % 111
			& $2.81 \times 10^{-1}$  % 110   
			& $3.10 \times 10^{-3}$  % 101
			& $2.81 \times 10^{-1}$  % 100
			& $3.10 \times 10^{-3}$  % 011
			& $2.81 \times 10^{-1}$  % 010   
			& $3.10 \times 10^{-3}$  % 001
			& $2.81 \times 10^{-1}$  % 000
			\\
			Objective achieved
			& \adjustbox{bgcolor=pink}{$3.10 \times 10^{-3}$}  % 111
			& $3.12 \times 10^{-3}$  % 110   
			& \adjustbox{bgcolor=pink}{$3.10 \times 10^{-3}$}  % 101
			& $3.11 \times 10^{-3}$  % 100
			& \adjustbox{bgcolor=pink}{$3.10 \times 10^{-3}$}  % 011
			& $3.13 \times 10^{-3}$  % 010   
			& \adjustbox{bgcolor=pink}{$3.10 \times 10^{-3}$}  % 001
			& $3.12 \times 10^{-3}$  % 000
			\\
			\midrule
			&\multicolumn{8}{c}{\texttt{HELMHOLTZ}} \\
			Running time [all,h] 
			& 21.66 % 111
			& 10.4  % 110   
			& 11.3  % 101
			& 9.29  % 100
			& 35.42 % 011
			& 8.74  % 010   
			& 41.25 % 001
			& 8.62  % 000
			\\
			Running time [heur.,h] 
			& 6.63 % 111
			& 0    % 110   
			& 6.62 % 101
			& 0    % 100
			& 6.64 % 011
			& 0    % 010   
			& 6.61 % 001
			& 0    % 000
			\\
			Cum.\ inner iterations 
			& 260          % 111
			& 1000 (limit) % 110   
			& 145          % 101
			& 1000 (limit) % 100
			& 277          % 011
			& 1000 (limit) % 010   
			& 263          % 001
			& 1000 (limit) % 000
			\\
			Accepted steps 
			& 89  % 111
			& 475 % 110   
			& 52  % 101
			& 467 % 100
			& 91  % 011
			& 460 % 010   
			& 87  % 001
			& 466 % 000
			\\			
			Initial objective
			& $2.34 \times 10^{-2}$  % 111
			& $2.92 \times 10^{-1}$  % 110   
			& $2.34 \times 10^{-2}$  % 101
			& $2.92 \times 10^{-1}$  % 100
			& $2.34 \times 10^{-2}$  % 011
			& $2.92 \times 10^{-1}$  % 010   
			& $2.34 \times 10^{-2}$  % 001
			& $2.92 \times 10^{-1}$  % 000
			\\			
			Objective achieved
			& \adjustbox{bgcolor=pink}{$2.25 \times 10^{-2}$}  % 111
			& $2.38 \times 10^{-2}$  % 110   
			& $2.26 \times 10^{-2}$  % 101
			& $2.39 \times 10^{-2}$  % 100
			& \adjustbox{bgcolor=pink}{$2.25 \times 10^{-2}$}  % 011
			& $2.40 \times 10^{-2}$  % 010   
			& \adjustbox{bgcolor=pink}{$2.25 \times 10^{-2}$}  % 001
			& $2.40 \times 10^{-2}$  % 000
			\\			
			\bottomrule
		\end{tabular}
	\end{adjustbox}
\end{table}

\section{Performance Improvement for \texorpdfstring{\ref{alg:patch_slip}}{Patch SLIP}}\label{sec:patch_slip}
Now we state and describe improvements for 
the patch version of \ref{alg:slip} from \cite{baraldi2025domain}, 
given in \ref{alg:patch_slip}. 
As before, we describe the algorithm itself, 
then detail the two performance improvements
of various patch shapes/sizes and 
using only the greedy step. 
We then discuss performance
of the tests on the benchmark problems 
compared to \ref{alg:slip}.

\subsection{Patch SLIP}
\ref{alg:patch_slip} was recently proposed in \cite{baraldi2025domain}
for solving \eqref{eq:p} via breaking the domain
$\Omega$ into patches and optimizing the patches separately,
thereby inducing cheaper subproblems than \ref{alg:slip}.
Numerical results in \cite{baraldi2025domain} detail speedups of over 125x 
on certain numerical examples over its counterpart \ref{alg:slip}. 
\ref{alg:patch_slip} separates the patch updates into 
acceptable candidates $\cA$ and working candidates $\cW$. 
The inner loop, in which the patches are updated, 
iterates over the working set $\cW$; once $\cW$ is empty, we exit the inner loop.
Within the inner loop,
trial iterates solved over each patch/radius combination yield predicted
and actual reductions. 
If  $\pred^{n,k,D} + \Delta_0 2^{-k}$ for a specific patch is greater 
than the maximum $\ared^{n,k,D}$ over all $D$, 
it is added to the working set over which we continually improve 
that block.
If the block update sufficiently decreases
the cost function by some fraction of the predicted reduction, 
then it is added to the acceptable set $\cA$ and removed from the  
working set $\cW$. 
Once $\cW$ is empty when either $\pred^{n,k,D}$ is zero and hence
stationary or satisfies acceptable step criteria, we then terminate the inner loop of the
algorithm. 
Afterwards, we enter the update loop in which we either 
test whether all block updates decrease the cost function 
one-by-one in descending order from the $\ared^{n,k,D}$
computed from the inner loop; this was the aforementioned ``greedy'' step. 
We propose a heuristic to speed up this computation below. 
The algorithm \ref{alg:patch_slip} is reproduced here 
with the modification from \cite{baraldi2025domain}. 

\begin{algorithm}[htb]
	\def\thealgorithm{Patch-SLIP}
	\refstepcounter{algorithm}
	\caption*{\textbf{Algorithm Patch-SLIP: Sequential linear integer programming method
			with greedy patch updates}}\label{alg:patch_slip}
	
	\textbf{Input:} $F$ satisfying \cref{ass:standing} with smoothness constant $L_{\nabla F}$,
	$\Delta^0 > 0$, $w^0 \in \BVW(\Omega)$, $\sigma \in (0,1)$, set of patches $\calD$,
	flag \texttt{ONLY\_USE\_GREEDY\_STEP} $ \in \{\texttt{T},\texttt{F}\}$.
	\begin{algorithmic}[1]
		\For{$n = 0,1,2,\ldots$}
		\State Set $\calA \gets \emptyset$, $\calW \gets \{ (0,D) \,|\, D \in \calD \}$.
		\For{$k = 0,1,2,\ldots$}\label{ln:tabulation_loop}
		\While{$(k,D) \in \calW$}\label{ln:inner_tabulation_loop}
		\State $\tilde{w}^{n,k,D} \gets$
		minimizer of $\text{{\ref{eq:trp}}}(w^{n}, \nabla F(w^{n}), D, \Delta^{0} 2^{-k})$. \label{ln:parallel_trstep}
		\State $\pred^{n,k,D} \gets
		(\nabla F(w^{n}), w^{n} - \tilde{w}^{n,k,D})_{L^2}
		+ \alpha \TV(w^{n}) - \alpha \TV(\tilde{w}^{k,n,D})$\label{ln:parallel_pred}
		\State $\ared^{n,k,D} \gets F(w^{n}) + \alpha \TV(w^{n})
		- F(\tilde{w}^{n,k,D}) - \alpha\TV(\tilde{w}^{n,k,D})$
		\If{$\ared^{n,k,D} \ge \sigma \pred^{n,k,D}$ \textbf{and} $\pred^{n,k,D} > 0$}\label{ln:suff_dec}
		\State $\calA \gets \calA\cup \{ (k,D) \}$.
		\ElsIf{$\pred^{n,k,D} > 0$ \textbf{and} $\max_{(\tilde{k},\tilde{D}) \in \calA} \ared^{n,\tilde{k},\tilde{D}}
			< \pred^{n,k,D} + L_{\nabla F} \Delta^{0}
			2^{-k}$}\label{ln:not_pred_zero_or_ared_domination}
		\State $\calW \gets \calW \cup \{(k + 1,D)\}$\label{ln:increase_k}
		\EndIf
		\State $\calW \gets \calW \setminus \{(k,D)\}$.
		\EndWhile
		\If{$\calW= \emptyset$}
		\State \textbf{break}
		\EndIf
		\EndFor
		\If{$\calA = \emptyset$}
		\State \textbf{return} $w^n$ is stationary.
		\EndIf
		\State $\bar{w}^n \gets w^n$.
		\State $\bar{j}^0 \gets F(w^n) + \alpha \TV(w^n)$.
		\While{$\calA \neq \emptyset$}\label{ln:calA_while_loop}
		\State $\bar{k},\bar{D} \gets \argmax\{\ared^{n,k,D}\,|\,(k,D) \in \calA \}$\label{ln:greedy}
		\State $\tilde{w}^n \gets \bar{w}^n\chi_{\Omega\setminus \bar{D}}
		+ \chi_{\bar{D}}\tilde{w}^{n,\bar{k}, \bar{D}}$
		\State $\bar{j} \gets F(\tilde{w}^n) + \alpha \TV(\tilde{w}^n )$
		\If{$\bar{j} < \bar{j}^0$}\label{ln:apply_reduction}
		\State $\bar{w}^n \gets \tilde{w}^n$, $\bar{j}^0 \gets \bar{j}$,
		$\calA \gets \calA\setminus \{ (\bar{k},\bar{D}) \}$.
		\EndIf
		\If{$\bar{j} \ge \bar{j}^0$ or \texttt{ONLY\_USE\_GREEDY\_STEP}}
		\State \textbf{break}
		\EndIf
		\EndWhile
		\State $w^{n + 1} \gets \bar{w}^n$
		\EndFor
	\end{algorithmic}
\end{algorithm}

\subsection{Performance Improvements}
For \cref{alg:patch_slip}, we assess the influence of two different degree of freedom on the performance,
namely the patch layout and optional improvements in \ref{alg:patch_slip} over the greedy step that may
be activated by setting \texttt{ONLY\_USE\_GREEDY\_STEP = F}.

\subsubsection{Patch Layout}\label{sec:domain_decomposition_mechanisms}
In \cite{baraldi2025domain} the domain was decomposed into square-shaped patches of equal size. We propose to compare this approach
to a decomposition into rectangles or stripes. The main result in \cite{manns2025discrete} showed the existence of the so-called fractional component
for the linear programming relaxation to the integer linear subproblems. It was argued that there seems to be a correlation between the size of
the fractional component in the root linear program and the computational demand of the integer program. This theoretically motivates to uses squares
and smaller patches in general. Specifically, if we use squares, we have a relatively small ratio volume to boundary ratio, which would be reflected
in the total variation and thus, our goal of a decomposition into squares is to either obtain smaller fractional components or fractional components
which have a larger volume to boundary ratio, which should result in smaller duality gaps; see also the theory on the cutting plane-based improvement.
On the other hand, a rectangular domain with a small and a long axis would require less fixed variables to turn the two-dimensional subproblems into
one-dimensional ones. The latter are computationally much less expensive and could even be solved with a dynamic programming approach; see 
\cite{severitt2023efficient}.

\subsubsection{Improvement over Greedy Step}
If $\calA \neq \emptyset$ in \ref{alg:patch_slip}, the combination $(k,D) \in \calA$ that realizes the largest actual reduction is the
greedy element and always used to improve the objective in the computation of $\bar{w}^n$ to improve over $w^n$. Note that $\bar{j} \ge \bar{j}^0$
always holds in the first execution of the while-loop starting at Line \ref{ln:calA_while_loop}. If \texttt{ONLY\_USE\_GREEDY\_STEP = T}, then
the $n$-loop iteration is completed and the algorithm moves on to the next iteration. However, this means that in most cases much information
is discarded since instances of \eqref{eq:trp} have been solved on all patches. Because the overlap between the patches
may be relatively small, it seems likely that at least the other elements in $\calA$ may improve the objective over the greedy step
by changing the control in different parts of $\Omega$. Therefore, we proceed as follows if \texttt{ONLY\_USE\_GREEDY\_STEP = F}. We
remove the greedy step from $\calA$ and take the element of $\calA$ that now has the largest actual reduction. We check if adding
this step to the trial control decreases the objective further; if affirmative, we accept this addition, remove this element from $\calA$
and start anew. If it does not decrease the objective further, we accept the previous trial control and the $n$-loop iteration is completed.
In this way, information of the subproblem solutions that led not to the selected greedy element in $\calA$ but to other suitable greedy candidates 
is not discarded, but instead used to heuristically improve the step over the greedy element. We note that this improvement strategy
was activated in all computations in the previous article \cite{baraldi2025domain}.

\subsection{Numerical Assessment}
We assess the performance of \ref{alg:patch_slip} on all of our
benchmark problems using five different patch layouts,
namely, $2 \times 2$, $3 \times 3$, and $4 \times 4$ squares
and $1 \times 4$ and $4 \times 1$ rectangles. The overlap between 
neighboring patches is always $10\,\%$ in every coordinate axis.
In addition, we assess the effect of the 
\texttt{ONLY\_USE\_GREEDY\_STEP = T/F} in \ref{alg:patch_slip}
so that we obtain $10$ executions of \ref{alg:patch_slip}
per benchmark problem in this assessment.
Due to the findings of \ref{alg:slip} in \cref{sec:slip}, we use
\texttt{Lazy = T}, \texttt{Cuts = T}, and \texttt{Early = T}
for all executions. We use \texttt{Heur.\ = T} for the benchmark
problems \texttt{AD}, \texttt{CHOUPI}, and \texttt{HELMHOLTZ}.
Since \texttt{Heur.\ = T} implies that \ref{alg:patch_slip} would
be initialized extremely close or at the solution 
to benchmark problem \texttt{EXACT}, we set \texttt{Heur.\ = F}
for \texttt{EXACT}. We prescribe a limit of $1000$
outer iterations and a time limit of $72$ hours. This makes eight 
settings in total that are tested against each other for all 
of the four benchmark problems. 

As for \ref{alg:slip} in \cref{sec:slip_test_config_2}, we obtain that 
\ref{alg:patch_slip} reaches the prescribed time limit for almost
all of the  assessed settings for benchmark problems \texttt{AD}, \texttt{CHOUPI}, 
\texttt{EXACT}, where problem \texttt{AD} with \texttt{ONLY\_USE\_GREEDY\_STEP = T}
and a $4 \times 4$ patch layout is the only exception (terminating after contraction
of the trust-region radius after $51.71$ hours. It terminates for benchmark problem \texttt{HELMHOLTZ}
due to a contraction of the trust-region radius for all patch layouts if \texttt{ONLY\_USE\_GREEDY\_STEP = F}
and for all patch layouts except $4 \times 4$ if \texttt{ONLY\_USE\_GREEDY\_STEP = T}.
While there are a few exceptions, the achieved objective values are generally lower 
and the compute times to reach the best objective value achieved by the \ref{alg:slip}
are generally higher if \texttt{ONLY\_USE\_GREEDY\_STEP = F}. In particular,
the best objective value achieved by \ref{alg:slip} is not achieved by 
\ref{alg:patch_slip} in half of the cases if \texttt{ONLY\_USE\_GREEDY\_STEP = F}
while this only happens three times if \texttt{ONLY\_USE\_GREEDY\_STEP = T}.
Consequently, we assess that it makes sense to include instead of discard
the information from the subproblems that did not maximize the actual
reduction in \ref{alg:patch_slip} and to set \texttt{ONLY\_USE\_GREEDY\_STEP = F}
in further experiments. We thus determine the best-performing patch layouts for
the setting \texttt{ONLY\_USE\_GREEDY\_STEP = F}. To this end,
we consider the benchmark problems one-by-one.

\paragraph{Benchmark \texttt{AD}}
The achieved objective value coincides for all five patch layouts
and for the execution of \ref{alg:slip} with the same performance
improvements even up to six digits of precision. In addition, the last
of the generated \eqref{eq:trp}-instances takes most of the
given run time window without completing. We have observed this situation
frequently when the optimization has effectively completed and the
last instance has a weak LP relaxation. Consequently, we find it sensible
to determine the performance of \ref{alg:patch_slip} based
on the time until the final objective value of \ref{alg:slip}
is reached within six digits of precision. The run time for this
is by far the lowest for the patch $1 \times 4$, which only takes
$2.60$ hours.

\paragraph{Benchmark \texttt{CHOUPI}}
The achieved objective values are all lower than for \ref{alg:slip}
with the rectangular patch layouts $1 \times 4$ and $4 \times 1$ both 
achieving $3.03 \times 10^{-3}$ compared to $3.04 \times 10^{-3}$ for
$4 \times 4$ and $3.05 \times 10^{-3}$ for $2 \times 2$ for $3 \times 3$.
In addition, the compute times to reach the achieved objective
value of the corresponding run of \ref{alg:slip} are significantly lower
for these two patch configurations, $32.61$ and $31.69$ hours
for $1 \times 4$ and $4 \times 1$ compared to $68.92$, $61.05$,
and $54.20$ for $2\times 2$, $3 \times 3$, and $4 \times 4$.
Due to this shorter time, we will use this one for comparisons to
\ref{alg:slip} and a randomized variant of \ref{alg:patch_slip} 
further down below but emphasize that we refrain from assessing
a clear best-performing patch configuration for this benchmark problem.

\paragraph{Benchmark \texttt{EXACT}}
The patch layout $3 \times 3$ is the only one to 
achieve the best objective value of $3.10\times 10^{-3}$.
This is the value of achieved by rounding of the already almost 
binary-valued solution of the continuous, which is very close to the 
global optimum by construction of this benchmark problem;
see \cref{subsec:exact}.

\paragraph{Benchmark \texttt{HELMHOLTZ}}
All patch layouts except $4 \times 4$ achieve an objective value
of $2.25 \times 10^{-2}$. For all of the other four configurations,
the run times vary between $16.44$ and $33.56$ hours but the most
of the computate time is spent in the final iterations with tiny
in $w^n$ and the corresponding objective values. With caution
being advised here, we deem the patch layout $4 \times 1$ the winner
of the comparison since the best objective value for the corresponding 
configuration of \ref{alg:slip} is already been achieved here
in only $9.97$ hours, while the other layouts require $11.74$ hours 
($2 \times 2$) or $16.30$ hours ($1 \times 4$) to do so or do
not achieve ($3 \times 3$, $4\times 4$) the same value within a
precision of six digits.

\paragraph{Patch Benchmark Conclusions}
We report the data mentioned above, on which we have
based our assessment in \Cref{tbl:data_patch_slip}. 
We also note that we observe that the finest patch layout $4 \times 4$
gives the worst achieved objective values for all benchmark problems
except for \texttt{AD}, where all layouts yield the same value.
While we do not have enough evidence to make a definitive statement,
it may be the case that this patch layout is too fine to incorporate information
from enough parts of the domain to achieve convergence to stationary 
points of the same quality as the other layouts.
This may to some extent also be observed in a visual comparison of the final iterates
of \ref{alg:patch_slip} for the different patch layouts for benchmark \texttt{EXACT}
in \Cref{fig:exact_sol_cmp} (recall that the analytic solution of the undiscretized
problem is a circle).

\begin{figure}[ht]
	\centering
	\begin{subfigure}[b]{0.19\textwidth}
	\centering 
	{%
		\setlength{\fboxsep}{0pt}%
		\setlength{\fboxrule}{1pt}%
		\fbox{\includegraphics[width=\textwidth]{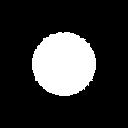}}
	}%
	\caption{{\small $2\times 2$}}
	\end{subfigure}
	\hfill
	\begin{subfigure}[b]{0.19\textwidth}
	\centering 
	{%
		\setlength{\fboxsep}{0pt}%
		\setlength{\fboxrule}{1pt}%
		\fbox{\includegraphics[width=\textwidth]{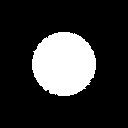}}
	}%
	\caption{{\small $3\times 3$}}
	\end{subfigure}
	\hfill
	\begin{subfigure}[b]{0.19\textwidth}
		\centering 
		{%
			\setlength{\fboxsep}{0pt}%
			\setlength{\fboxrule}{1pt}%
			\fbox{\includegraphics[width=\textwidth]{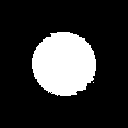}}
		}%
		\caption{{\small $4\times 4$}}
	\end{subfigure}
	\hfill
	\begin{subfigure}[b]{0.19\textwidth}  
		\centering 
		{%
			\setlength{\fboxsep}{0pt}%
			\setlength{\fboxrule}{1pt}%
			\fbox{\includegraphics[width=\textwidth]{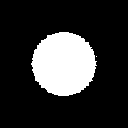}}
		}%
		\caption{{\small $1\times 4$}}
	\end{subfigure}
	\hfill
	\begin{subfigure}[b]{0.19\textwidth}  
		\centering 
		{%
			\setlength{\fboxsep}{0pt}%
			\setlength{\fboxrule}{1pt}%
			\fbox{\includegraphics[width=\textwidth]{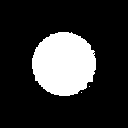}}
		}%
		\caption{{\small $4\times 1$}}
	\end{subfigure}
	\caption{Visualization of final iterate of
		\ref{alg:patch_slip} for \texttt{EXACT}
		for all patch layouts
		(black = $0$, white = $1$).}
	\label{fig:exact_sol_cmp}
\end{figure}

\begin{table}[h]
	\caption{Running time, time until the best objective produced
	by the corresponding configuration of \ref{alg:slip} is achieved, 
	number of completed iterations, and achieve objective value for
	\ref{alg:patch_slip} of \cref{sec:patch_slip} for different
	patch layouts. Winners in terms of objective value within the 
	reported precisioe are highlighted with pink color.
	The time to full domain objective was measured by
	comparing the objective values up to six digits of precision.}
	\label{tbl:data_patch_slip}
	\begin{adjustbox}{width=\textwidth}	
		\begin{tabular}{l|lllll}
			\toprule
			& $2 \times 2$ 
			& $3 \times 3$
			& $4 \times 4$
			& $1 \times 4$
			& $4 \times 1$
			\\
			&\multicolumn{5}{c}{\texttt{ONLY\_USE\_GREEDY\_STEP = F}} \\
			\midrule
			&\multicolumn{5}{c}{\texttt{AD}} \\
			Running time [h]
			& 72 (limit) % 2x2
			& 72 (limit) % 3x3   
			& 72 (limit) % 4x4
			& 72 (limit) % 3x3   
			& 72 (limit) % 4x1
			\\
			Time to final full domain obj.\ [h]
			& 8.00
			& 14.64
			& 10.58
			& 2.60
			& 7.69
			\\						
			Completed iterations 
			& 17
			& 10
			& 9
			& 14
			& 9
			\\
			Objective achieved
			& \adjustbox{bgcolor=pink}{$6.74 \times 10^{-1}$}
			& \adjustbox{bgcolor=pink}{$6.74 \times 10^{-1}$}
			& \adjustbox{bgcolor=pink}{$6.74 \times 10^{-1}$}
			& \adjustbox{bgcolor=pink}{$6.74 \times 10^{-1}$}
			& \adjustbox{bgcolor=pink}{$6.74 \times 10^{-1}$}
			\\			
			\midrule
			&\multicolumn{5}{c}{\texttt{CHOUPI}} \\
			Running time [h] 
			& 72 (limit) % 2x2
			& 72 (limit) % 3x3   
			& 72 (limit) % 4x4
			& 72 (limit) % 1x4
			& 72 (limit) % 4x1
			\\
			Time to final full domain obj.\ [h] 
			& $68.92$ % 2x2
			& $61.05$ % 3x3   
			& $54.20$ % 4x4
			& $32.61$ % 1x4
			& $31.69$ % 4x1
			\\			
			Completed iterations 
			& 33 % 2x2
			& 28 % 3x3
			& 24 % 4x4
			& 34 % 1x4
			& 34 % 4x1
			\\
			Objective achieved
			& $3.05 \times 10^{-3}$ % 2x2
			& $3.05 \times 10^{-3}$ % 3x3
			& $3.04 \times 10^{-3}$ % 4x4
			& \adjustbox{bgcolor=pink}{$3.03 \times 10^{-3}$} % 1x4
			& \adjustbox{bgcolor=pink}{$3.03 \times 10^{-3}$} % 1x4
			\\
			\midrule			
			&\multicolumn{5}{c}{\texttt{EXACT}} \\
			Running time [h] 
			& 72 (limit) % 2x2
			& 72 (limit) % 3x3
			& 72 (limit) % 4x4
			& 72 (limit) % 1x4
			& 72 (limit) % 4x1
			\\
			Time to final full domain obj.\ [h] 
			& $16.76$ % 2x2
			& $10.58$ % 3x3
			& not achieved
			& $5.12$  % 1x4
			& $27.79$ % 4x1
			\\			
			Completed iterations 
			& 16 % 2x2
			& 13 % 3x3
			& 12 % 4x4
			& 19 % 1x4
			& 17 % 4x1
			\\
			Objective achieved
			& $3.11 \times 10^{-3}$ % 2x2
			& \adjustbox{bgcolor=pink}{$3.10 \times 10^{-3}$} % 3x3
			& $3.13 \times 10^{-3}$ % 4x4
			& $3.11 \times 10^{-3}$ % 1x4
			& $3.11 \times 10^{-3}$ % 4x1
			\\
			\midrule
			&\multicolumn{5}{c}{\texttt{HELMHOLTZ}} \\
			Running time [h] 
			& $16.44$ % 2x2
			& $18.41$ % 3x3
			& $34.47$ % 4x4
			& $33.56$ % 1x4
			& $21.29$ % 4x1
			\\
			Time to final full domain obj.\ [h] 
			& $11.74$      % 2x2
			& not achieved % 3x3
			& not achieved % 4x4
			& $16.30$      % 1x4
			& $9.97$       % 4x1
			\\				
			Completed iterations 
			& 61 % 2x2
			& 45 % 3x3
			& 49 % 4x4
			& 95 % 1x4
			& 65 % 4x4
			\\
			Objective achieved
			& \adjustbox{bgcolor=pink}{$2.25 \times 10^{-2}$} % 2x2
			& \adjustbox{bgcolor=pink}{$2.25 \times 10^{-2}$} % 3x3
			& $2.26 \times 10^{-2}$ % 4x4
			& \adjustbox{bgcolor=pink}{$2.25 \times 10^{-2}$} % 1x4
			& \adjustbox{bgcolor=pink}{$2.25 \times 10^{-2}$} % 4x1
			\\
			\bottomrule 
			\toprule
			& $2 \times 2$ 
			& $3 \times 3$
			& $4 \times 4$
			& $1 \times 4$
			& $4 \times 1$
			\\
			&\multicolumn{5}{c}{\texttt{ONLY\_USE\_GREEDY\_STEP = T}} \\
			\midrule
			&\multicolumn{5}{c}{\texttt{AD}} \\
			Running time [h]
			& 72 (limit) % 2x2
			& 72 (limit) % 3x3   
			& 51.71 % 4x4
			& 72 (limit) % 1x4
			& 72 (limit) % 4x1
			\\
			Time to final full domain obj.\ [h] 
			& 6.26 % 2x2
			& 24.00 % 3x3
			& 12.79
			& 4.49 % 1x4
			& 13.97 % 4x1
			\\		
			Completed iterations 	
			& 19
			& 17
			& 18
			& 18
			& 18
			\\
			Objective achieved
			& \adjustbox{bgcolor=pink}{$6.74 \times 10^{-1}$}
			& \adjustbox{bgcolor=pink}{$6.74 \times 10^{-1}$}
			& \adjustbox{bgcolor=pink}{$6.74 \times 10^{-1}$}
			& \adjustbox{bgcolor=pink}{$6.74 \times 10^{-1}$}
			& \adjustbox{bgcolor=pink}{$6.74 \times 10^{-1}$}
			\\			
			\midrule
			&\multicolumn{5}{c}{\texttt{CHOUPI}} \\
			Running time [h] 
			& 72 (limit) % 2x2
			& 72 (limit) % 3x3   
			& 72 (limit) % 4x4
			& 72 (limit) % 1x4
			& 72 (limit) % 4x1
			\\
			Time to final full domain obj.\ [h] 
			& not achieved % 2x2
			& not achieved % 3x3   
			& not achieved % 4x4
			& not achieved % 1x4
			& 68.88 % 4x1
			\\			
			Completed iterations 
			& 48 % 2x2
			& 48 % 3x3
			& 52 % 4x4
			& 44 % 1x4
			& 47% 4x1
			\\
			Objective achieved
			& $3.05 \times 10^{-3}$
			& $3.05 \times 10^{-3}$
			& $3.05 \times 10^{-3}$
			& $3.07 \times 10^{-3}$
			& $3.05 \times 10^{-3}$
			\\
			\midrule			
			&\multicolumn{5}{c}{\texttt{EXACT}} \\
			Running time [h] 
			& 72 (limit) % 2x2
			& 72 (limit) % 3x3
			& 72 (limit) % 4x4
			& 72 (limit) % 1x4
			& 72 (limit) % 4x1
			\\
			Time to final full domain obj.\ [h] 
			& 56.03 % 2x2
			& 14.86 % 3x3
			& 19.03 % 4x4
			& not achieved % 1x4
			& not achieved % 4x1
			\\			
			Completed iterations 
			& 38 % 2x2
			& 40 % 3x3
			& 40 % 4x4
			& 25 % 1x4
			& 23 % 4x1
			\\
			Objective achieved
			& $3.11 \times 10^{-3}$
			& \adjustbox{bgcolor=pink}{$3.10 \times 10^{-3}$} % 3x3
			& $3.11 \times 10^{-3}$ % 4x4
			& $3.12 \times 10^{-3}$ % 1x4
			& $3.13 \times 10^{-3}$ % 4x1
			\\
			\midrule	
			&\multicolumn{5}{c}{\texttt{HELMHOLTZ}} \\
			Running time [h] 
			& 16.00 % 2x2
			& 31.67 % 3x3
			& 72 (limit) % 4x4
			& 25.90 % 1x4
			& 26.61 % 4x1
			\\
			Time to final full domain obj.\ [h] 
			& not achieved % 2x2
			& not achieved % 3x3
			& not achieved % 4x4
			& not achieved % 1x4
			& 13.75 % 4x1
			\\				
			Completed iterations 
			& 72 % 2x2
			& 107 % 3x3
			& 115 % 4x4
			& 94 % 1x4
			& 96 % 4x1
			\\
			Objective achieved
			& \adjustbox{bgcolor=pink}{$2.25 \times 10^{-2}$} % 2x2
			& \adjustbox{bgcolor=pink}{$2.25 \times 10^{-2}$} % 3x3
			& $2.26 \times 10^{-2}$
			& \adjustbox{bgcolor=pink}{$2.25 \times 10^{-2}$} % 1x4
			& \adjustbox{bgcolor=pink}{$2.25 \times 10^{-2}$} % 4x1
			\\								
			\bottomrule
		\end{tabular}
	\end{adjustbox}
\end{table}

\section{Randomized Patch SLIP}\label{sec:randomized_patch_slip}
In unconstrained optimization in $\R^n$, coordinate-descent methods are 
generally outperformed with respect to computational efficiency
by their randomized counterparts; the reasoning therein is that the latter
do not require the expensive greedy selection process for convergence 
guarantees. This motivates us to employ a trust-region algorithm
for solving \eqref{eq:p} that randomly selects patches from an open
cover of $\Omega$ and then solves
the trust-region subproblem on the selected
patch. While the convergence analysis is more involved in our setting, we
prove essentially the same asymptotics for our problem setting in 
\cref{sec:algorithm_analysis} below as one would prove for an
unconstrained problem in Euclidean space.

\subsection{Randomized Patch SLIP}
Our randomized SLIP variant proceeds
by using the current iterate $w^n \in \BVW(\Omega)$ to compute selection probabilities $p^n_D$ for $D \in \calD$, where $\calD$
is the open cover from \cref{ass:calD_open_cover}. Then a patch $D_n \in \calD$
is selected at random according to the just-computed probabilities.
Next, the trust-region subproblem
\eqref{eq:trp} (see \cite{baraldi2025domain,manns2023on})
is solved with the arguments $\bar{w} = w^n$, $g = \nabla F(w^n)$, $D = D_n$,
and $\Delta = \Delta_0 2^{-k}$ for increasing integers $k = 0, 1, \ldots$.

The inner iteration over increasing $k = 0,1,\ldots$ proceeds until one of two stopping criteria is met.
Either, the computed trial point can be accepted if the reduction of
the objective of \eqref{eq:p} $\ared{}^{n,k,D_n}$ that is realized by the
solution $\tilde{w}^{n,k,D_n}$ to \eqref{eq:trp} is at least a fraction of 
the predicted reduction $\pred{}^{n,k,D_n}$, and, in addition, $\pred{}^{n,k,D_n}$ exceeds
the threshold $\kappa_n$. The second stopping criterion
of the inner loop is that the predicted reduction 
$\pred{}^{n,k,D_n}$ falls below the threshold $\kappa_n$.
The algorithm is stated in pseudocode in \ref{alg:randomized_patch_slip}.
\begin{remark}
We highlight that \ref{alg:randomized_patch_slip} is a more simplistic algorithm than
\ref{alg:patch_slip}. A lot of the steps in \ref{alg:patch_slip} are made
to determine the greedy patch, which is necessary for our convergence proof
in \cite{baraldi2025domain}. Meanwhile the nature of \ref{alg:randomized_patch_slip}
simply requires updates in the randomly selected block, thereby forgoing the 
greedy update. 
\end{remark}

\begin{algorithm}[htb]
	\def\thealgorithm{Randomized-Patch-SLIP}
	\refstepcounter{algorithm}
	\caption*{\textbf{Algorithm Randomized-Patch-SLIP: Sequential linear integer programming method
			with random patch updates}}
	\label{alg:randomized_patch_slip}
	
	\textbf{Input:} $\kappa_0 > 0$, $\Delta_0 > 0$, $w^0 \in \BVW(\Omega)$, $\sigma \in (0,1)$, set of patches $\calD$.
	\begin{algorithmic}[1]
		\For{$n = 0,1,2,\ldots$}
		\State Assign selection probabilities $p_D^n$ for all $D \in \calD$ based on $w^n$.\label{ln:selection_probabilities}
		\State Pick $D_n \in \calD$ at random based on the probabilities $p_D^n$.\label{ln:select_patch}
		\For{$k = 0,1,\ldots$}\label{ln:tr_reduction}
		\State $\tilde{w}^{n,k,D} \gets$
		minimizer of $\operatorname{\ref{eq:trp}}(w^{n}, \nabla F(w^{n}), D_n, \Delta_{0} 2^{-k})$.\label{ln:tr_step}
		\State $\pred^{n,k,D_n} \gets (\nabla F(w^{n}), w^{n} - \tilde{w}^{n,k,D_n})_{L^2}
		+ \alpha \TV(w^{n}) - \alpha \TV(\tilde{w}^{k,n,D_n})$\label{ln:pred}
		\State $\ared^{n,k,D_n} \gets F(w^{n}) + \alpha \TV(w^{n})
		- F(\tilde{w}^{n,k,D_n}) - \alpha\TV(\tilde{w}^{n,k,D_n})$
		\If{$\ared^{n,k,D_n} \ge \sigma \pred^{n,k,D_n}$ and $\pred^{n,k,D_n} > \kappa_n$}\label{ln:psuff_dec}
		\State $w^{n+1} \gets \tilde{w}^{n,k,D_n}$
		\State $\kappa_{n+1} \gets \kappa_n$
		\State \textbf{go to} Line \ref{ln:selection_probabilities}.
		\ElsIf{$\pred^{n,k,D_n} \le \kappa_n$}\label{ln:reduce_kappa}
		\State $\kappa_{n+1} \gets 2^{-1}\kappa_n$
		\State \textbf{go to} Line \ref{ln:selection_probabilities}.		
		\EndIf
		\EndFor
		\EndFor
	\end{algorithmic}
\end{algorithm}
As the previous algorithms, \ref{alg:randomized_patch_slip} resets the
trust-region radius after a successful iterate, which is required
for the convergence analysis if $d \ge 2$ therein. We point the reader to \cite{friedemann2025trust,manns2024convergence}
for an analysis that is tied to $d = 1$ but allows to avoid the reset. The reset makes the algorithm closer to backtracking linesearch,
which is common for randomized coordinate descent methods.
Trust-region methods seldom appear in the literature; see \cite{wang2016randomized}
for an exception. In contrast to \ref{alg:slip}, we find it not intuitive whether a typical trust-region update strategy
is beneficial or not. Specifically, we are not aware of any arguments 
concerning why step acceptance on the
current coordinate block (or patch in our case) should make it likely that
the updated model performs well on another coordinate block (or patch) for a similarly-sized trust region.

Usually, optimization algorithms terminate when a non-negative function that can
be easily evaluated, is continuous, and measures instationarity (also called gap
function or criticality measure) \cite{conn1993global,larsson1994class} 
is zero or close to zero. In the unconstrained case, this is usually the norm of the
gradient of the objective. Similarly, in a block coordinate descent method, an update of a block of coordinates
may not be performed if the corresponding entries of the gradient are zero. Such a function
is not known for \eqref{eq:p} so far, however. Since it cannot be excluded that the
predicted reduction is strictly positive for all positive trust-region radii for a
patch on which \eqref{eq:stationarity} is satisfied, we need another means to terminate
the inner loop in \ref{alg:randomized_patch_slip}. We do so by means of the threshold $\kappa_n$.
If the predicted reduction for an inner iteration $k$ falls below $\kappa_n$, the inner loop
is terminated, $\kappa_n$ is halved, and the patch selection is begins anew.
This implies that the inner loop always terminates.

\begin{remark}
	For implementations of \ref{alg:randomized_patch_slip}, where the space $\BVW(\Omega)$ is
	discretized using piecewise constant functions, one does not need to implement
	the aforementioned mechanism using $\kappa_n$ but one can just stop the inner loop if the
	trust-region falls below the smallest $d$-dimensional volume of all grid cells because in this
	case, the integrality restriction precludes the existence of feasible trial points other than
	the current iterate. Clearly, $p_{n+1}^{D_n}$ can be set to zero when this happens and
	the whole algorithm can be terminated if this contraction happens for all patches
	while $w_n$ remains unchanged.
\end{remark}

\subsection{Statement and Discussion
	of Convergence Result}
A key ingredient in the algorithm is the computation of the patch selection probabilities
in \ref{alg:randomized_patch_slip} ll.\ \ref{ln:selection_probabilities},\ref{ln:select_patch}. We make the
standing \cref{ass:selection_probabilities}
below on the probabilities for patch selections, which will turn out
to be sufficient to achieve a convergence result that parallels the usual \emph{convergence in
	expectation} for block coordinate descent methods.

\begin{remark}
	The typical convergence result in randomized coordinate
	descent methods is that the iterates converge to stationary points in expectation; see, e.g., \cite{wright2015coordinate}.
	For the simple case of the smooth, unconstrained optimization of
	$f : \R^d \to \R$, this amounts
	to proving $\E \|\nabla f(X^n)\| \to 0$ as $n \to \infty$ if the $X^n$
	are the random variables for the algorithm iterates. As a consequence, $\nabla f(x) = 0$ holds for any limit point $x$ of any realization
	of $\{X^n\}_n$ with probability one, that is, every limit point is stationary with probability one.
	This argument is not directly applicable here since we do not have a surrogate for
	the function $x \mapsto \|\nabla f(x)\|$ for $d \ge 2$. Nevertheless, every subsequence of iterates is bounded by the problem setting and algorithm design and admits weak-$^*$ cluster points in $\BV(\Omega)$.
	We therefore analyze the probability of the event that there is a weak-$^*$ cluster point of iterates that is
	not stationary and show that it is indeed zero so that our result can be interpreted as convergence to stationary points
	almost surely or convergence in expectation.
\end{remark}

\begin{assumption}\label{ass:selection_probabilities}
	We assume that the probability assignment in \ref{alg:randomized_patch_slip} l.\ \ref{ln:selection_probabilities} is determined
	functionally on $w^n$ such that
	\begin{enumerate}
		\item $p_D^n \ge 0$ for all $D \in \calD$ and $\sum_{D \in \calD} p_D^n = 1$.
		\item there exist $p_0 > 0$ and $\kappa_n \searrow 0$ such that for all $n \in \N$ and $D \in \calD$, 
		$p_0 \le p_D^n$ and if
		\[ w^n\text{ is not patch-stationary on }D\quad\text{or}\quad \pred{}^{n,0,D} > \kappa_n.\]
	\end{enumerate}
\end{assumption}
\Cref{ass:selection_probabilities} covers the often-implemented special case of a uniform probability for all patches for all iterations.
\begin{lemma}\label{lem:iid}
	Let the probabilities assigned in \ref{alg:randomized_patch_slip} l.\ \ref{ln:selection_probabilities}
	satisfy $p_D^n = p_D = \frac{1}{|\calD|}$ for all iterations $n \in \N$.
	Then \cref{ass:selection_probabilities} holds.
\end{lemma}
\begin{proof}
	This is immediate.
\end{proof}
\begin{remark}
	\Cref{ass:selection_probabilities} gives flexibility over the situation of \cref{lem:iid}
	for heuristic choices that allow to adapt the probabilities to the current situation, for example, because
	good progress on a patch is often followed by good progress on the same patch in the next iteration.
	Similarly, a patch for which a priori knowledge is available that it cannot improve the objective
	much can be excluded from consideration by setting $p_{n}^D$ to zero or a value close to zero.
\end{remark}

We now state our main convergence
result and defer the proof to the appendix in
\cref{sec:algorithm_analysis}.

\begin{restatable}{theorem}{mainthm}\label{thm:main_convergence_theorem}
	Let $(X, \Sigma_P, P)$ be the probability space constructed in
	\cref{sec:probability_space}.
	Let \cref{ass:calD_open_cover,ass:standing,ass:selection_probabilities} hold.
	Let
	\[ N \coloneqq 
	\{ x \in X\,:\,
	\exists (n_\ell)_\ell \subset (n)_n : W_{n_\ell}(x) \weakstarto \bar{w} 
	\text{ in } \BVW(\Omega)
	\text{ for some instationary } \bar{w}
	\}.
	\]
	Then $N \in \Sigma_P$ and $P(N) = 0$.
\end{restatable}

%
%\begin{remark}\label{rem:trp_stationary_points}
%	Due to the exact consideration of the term $\TV(w)$,
%	it can happen that $\pred^{n,k,D_n} > 0$ holds and $\tilde{w}^{n,k,D_n}$ is accepted
%	when $w^n$ is stationary on $D_n$. In this case, the algorithm improves over (suboptimal)
%	locally stationary point $w^n$ and proceeds towards a (better) stationary point. On the other
%	hand, if $\pred^{n,k,D_n} = 0$ holds, then $w^n$ solves
%	$\operatorname{\ref{eq:trp}}(w^n, \nabla F(w^n), D_n, \Delta_0 2^{-k})$ and is in turn
%	stationary on $D$; see Proposition 5.5 in \cite{manns2023on} (the proof directly carries over
%	to the restriction to $D$ by replacing $\phi \in C_c^\infty(\Omega; \R^n)$ therein by
%	$\phi \in C_c^\infty(D; \R^n)$. Clearly, we also cannot exclude that $w^n$ is stationary
%	on $D$ but the acceptance criterion $\ared^{n,k,D_n} \ge \sigma \pred^{n,k,D_n}$
%	is never satisifed. See Corollary 6.3 in \cite{manns2023on} for the analogous observation
%	for SLIP.
%\end{remark}

\subsection{Numerical Assessment}
We now assess the performance of \ref{alg:randomized_patch_slip} in 
comparison to \ref{alg:slip} and \ref{alg:patch_slip}. Again, we activate 
early termination, the lazy constraint mechanism, and the cutting
planes for all four test cases and the relax-\&-round initalization
for all test cases except \texttt{EXACT}. Regarding
the patch layout for our test cases, we use the best performing ones from the
previous section, that is, $1 \times 4$ for \texttt{AD}, $4 \times 1$ for
\texttt{CHOUPI}, $3 \times 3$ for \texttt{EXACT}, and $4 \times 1$ for
\texttt{HELMHOLTZ}. Since the inclusion of objective-decreasing
steps (flag \texttt{ONLY\_USE\_GREEDY\_STEP = F}) over the
greedy step in one iteration lead to a substantial improvement of the 
performance of \ref{alg:patch_slip}, we use this variant
for comparison here. Again, we prescribe a time limit of $72$ hours.

We note that, as we will see below, the differences between the achieved
objective values for \ref{alg:patch_slip} and \ref{alg:randomized_patch_slip} are very small
so that we have decided to report them up to 5 digits of precision for this test case but advice caution
in putting too much importance on the differences.

\paragraph{Benchmark \texttt{AD}} The assessed algorithm variants have 
converged to almost the same objective value and iterate
within the time limit; see also \Cref{fig:ad_sol_cmp}.
At the end, a computationally very expensive subproblem is generated that cannot
be solved within the given time limit. Again, we note that this often happens when the fractional 
component of the LP relaxation coincides with the whole domain or the whole patch, as
in the discussion  \cref{sec:slip_test_config_2}.
In this situation, the randomized \ref{alg:randomized_patch_slip} may be
deemed as performing best because the final iteration is achieved after in less
than half of the time ($14.60$ hours) compared to \ref{alg:slip} ($32.18$ hours) and \ref{alg:patch_slip} ($46.38$ hours). We note that, as we have observed in the test before, the last iterations only alter the objective very little and
if one considers the time until the final objective value is achieved
up to a precision of $10^{-6}$, we obtain $24.13$ hours for 
\ref{alg:slip}, $10.76$ hours for \ref{alg:patch_slip},
and $7.21$ hours \ref{alg:randomized_patch_slip},
which is consistent with deeming \ref{alg:randomized_patch_slip} the winner for this test case.
\begin{figure}[ht]
	\centering
	\begin{subfigure}[b]{0.31\textwidth}
		\centering 
		{%
			\setlength{\fboxsep}{0pt}%
			\setlength{\fboxrule}{1pt}%
			\fbox{\includegraphics[width=\textwidth]{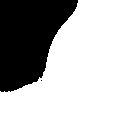}}
		}%
		\caption{{\small \ref{alg:slip}}}
	\end{subfigure}
	\hfill
	\begin{subfigure}[b]{0.31\textwidth}  
		\centering 
		{%
			\setlength{\fboxsep}{0pt}%
			\setlength{\fboxrule}{1pt}%
			\fbox{\includegraphics[width=\textwidth]{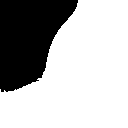}}
		}%
		\caption{{\small \ref{alg:patch_slip}}}
	\end{subfigure}
	\hfill
	\begin{subfigure}[b]{0.31\textwidth}  
		\centering 
		{%
			\setlength{\fboxsep}{0pt}%
			\setlength{\fboxrule}{1pt}%
			\fbox{\includegraphics[width=\textwidth]{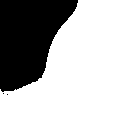}}
		}%
		\caption{{\small \ref{alg:randomized_patch_slip}}}
	\end{subfigure}
	\caption{Visualization of final iterate of
		\ref{alg:slip},
		\ref{alg:patch_slip}, and
		\ref{alg:randomized_patch_slip} for \texttt{AD}
		(black = $0$, white = $1$).}
	\label{fig:ad_sol_cmp}
\end{figure}

\paragraph{Benchmark \texttt{CHOUPI}} We observe that the differences 
between the time to the last accepted iterate and the prescribed time 
limit are comparatively small (all of them less than two hours), which 
indicates that the optimization was not effectively completed on 
termination. We observe that \ref{alg:patch_slip} achieves a lower 
objective value \ref{alg:slip} and \ref{alg:randomized_patch_slip}
within the prescribed time limit. Specifically, the initial iterate
(from the relaxation-\&-round initialization) has an objective
value of $3.4391 \times 10^{-3}$, which was decreased to
$3.0268 \times 10^{-3}$ by \ref{alg:patch_slip} in comparison to
$3.0481 \times 10^{-3}$ ($5.2$\,\% lower) by \ref{alg:slip} and
$3.0288 \times 10^{-3}$ ($0.49$\,\% lower) for
\ref{alg:randomized_patch_slip}. The achieved objective values over
time are plotted in \Cref{fig:choupi_performance_plot},
where we have excluded the time for the
relax-\&-round initialization because its time consumption and
objective reduction is the same for all algorithms. Since the
relative difference between
the objective values for \ref{alg:patch_slip} and 
\ref{alg:randomized_patch_slip} is relatively small, there is no
clear winner here but we find it fair to conclude that these two
algorithms slightly outperform \ref{alg:slip}.
\begin{figure}	
	\begin{tikzpicture}
	\begin{axis}[width=\textwidth,height=7cm,
	xlabel={Compute time},
	ylabel={Objective},
	grid=both,
	legend style={at={(0.98,0.98)},anchor=north east},
	]
	
	\addplot[
	solid,
	mark=*,
	mark size=2.5pt,
	]
	table {data/choupi11111_pp.txt};
	\addlegendentry{\ref{alg:slip}}
	
	\addplot[
	dashed,
	mark=square*,
	mark size=2.5pt,
	]
	table {data/choupi11111_41_pp.txt};
	\addlegendentry{\ref{alg:patch_slip}}
	
	\addplot[
	dotted,
	mark=triangle*,
	mark size=2.5pt,
	]
	table {data/choupi11111_41_r_pp.txt};
	\addlegendentry{\ref{alg:randomized_patch_slip}}
	
	\end{axis}
	\end{tikzpicture}
	\caption{Objective value over time for benchmark 
		\texttt{CHOUPI} for \ref{alg:slip}, \ref{alg:patch_slip}, and
		\ref{alg:randomized_patch_slip}.}\label{fig:choupi_performance_plot}
\end{figure}
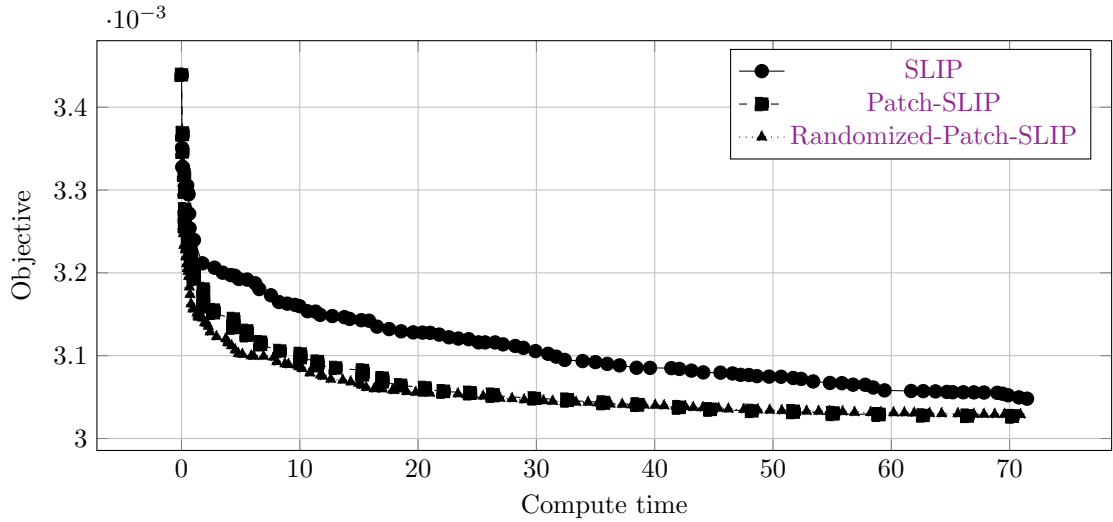

\paragraph{Benchmark \texttt{EXACT}} At first glance, the situation is 
similar to the testcase \texttt{AD} in that for all three assessed 
algorithms, a subproblem is generated that has a very long compute time 
and no further progress is made. This observation is not as pronounced 
\ref{alg:randomized_patch_slip} as for the other algorithms because
the last iteration has \emph{only} consumed a little less than
10 hours so far so that one might argue that it has not yet completed. 
However, the remaining duality gap of the last generated instance
of \eqref{eq:p} is only decreasing at a very slow pace and with
remaining (relative) duality gap of $0.43\,\%$ still sitting
well above the termination threshold of $0.01\,\%$. Considering
the evolution of the objective over time, one can observe that
\ref{alg:slip} decreases the objective relatively fast and in particular 
faster than \ref{alg:randomized_patch_slip} until the \emph{intractable}
subproblem is generated. We note that by construction, solving the 
continuous relaxation, we obtain a almost binary-valued solution due
to the construction of \texttt{EXACT}, which after rounding to integers 
yields the objective value $3.1022 \times 10^{-3}$, equally good
as the final iterate produced by \ref{alg:patch_slip}.
Therefore, we deem \ref{alg:patch_slip} the winner in terms of objective 
achieved but highlight that the large, inexpensive steps that can be made
by \ref{alg:slip} by not being restricted to patches clearly give an
edge at the beginning of the algorithm. This is of course mainly
the case here because we do not initialize with relaxation-\&-round
here, which we strongly recommend however given the other results.
\begin{figure}
	\centering
	\begin{subfigure}[b]{0.49\textwidth}
		\centering 
		\begin{tikzpicture}
		\begin{axis}[width=\textwidth,height=8cm,
		xlabel={Compute time},
		ylabel={Objective},
		grid=both,
		legend style={at={(0.98,0.98)},anchor=north east},
		]
		
		\addplot[
		solid,
		mark=*,
		mark size=2.5pt,
		restrict x to domain=0:1,
		]
		table {data/exact11110_pp.txt};
		\addlegendentry{\ref{alg:slip}}
		
		\addplot[
		dashed,
		mark=square*,
		mark size=2.5pt,
		restrict x to domain=0:1,
		]
		table {data/exact11110_33_pp.txt};
		\addlegendentry{\ref{alg:patch_slip}}
		
		\addplot[
		dotted,
		mark=triangle*,
		mark size=2.5pt,
		restrict x to domain=0:1,
		]
		table {data/exact11110_33_r_pp.txt};
		\addlegendentry{\ref{alg:randomized_patch_slip}}
		
		\end{axis}
		\end{tikzpicture}
		\caption{First hour of compute time}
	\end{subfigure}
	\hfill
	\begin{subfigure}[b]{0.49\textwidth}  
		\centering 
		\begin{tikzpicture}
		\begin{axis}[width=\textwidth,height=8cm,
		xlabel={Compute time},
		ylabel={Objective},
		grid=both,
		legend style={at={(0.98,0.98)},anchor=north east},
		xmin=0.5,
		xmax=72
		]
		
		\addplot[
		solid,
		mark=*,
		mark size=2.5pt,
		restrict x to domain=1:72,
		]
		table {data/exact11110_pp.txt};
		\addlegendentry{\ref{alg:slip}}
		
		\addplot[
		dashed,
		mark=square*,
		mark size=2.5pt,
		restrict x to domain=1:72,
		]
		table {data/exact11110_33_pp.txt};
		\addlegendentry{\ref{alg:patch_slip}}
		
		\addplot[
		dotted,
		mark=triangle*,
		mark size=2.5pt,
		restrict x to domain=1:72,
		]
		table {data/exact11110_33_r_pp.txt};
		\addlegendentry{\ref{alg:randomized_patch_slip}}
		
		\end{axis}
		\end{tikzpicture}
		\caption{{\small Further compute time}}
	\end{subfigure}

	\caption{Objective value over time for testcase 
		\texttt{EXACT} for 
		\ref{alg:slip}, \ref{alg:patch_slip}, and
		\ref{alg:randomized_patch_slip}.}\label{fig:exact_performance_plot}
\end{figure}
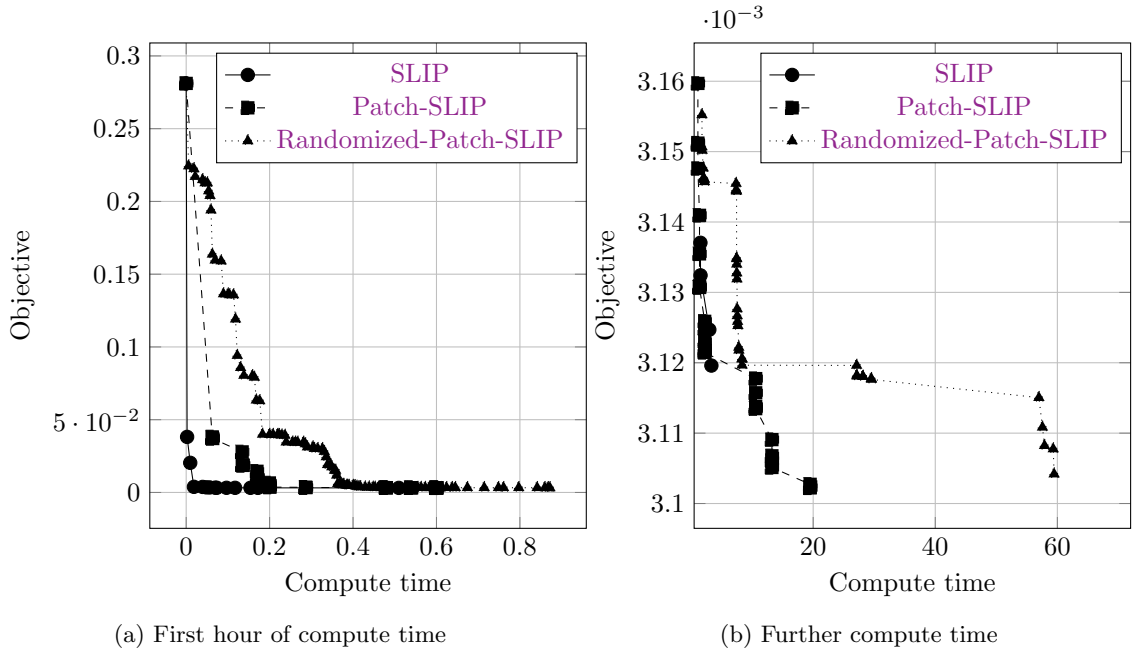

\paragraph{Benchmark \texttt{HELMHOLTZ}} The assessed algorithms
have completed their computation within the given time limit.
\ref{alg:slip} has the longest runtime of $21.66$ hours
compared to $21.29$ hours for \ref{alg:patch_slip} and
$13.13$ hours for \ref{alg:randomized_patch_slip}.
It also shows the longest time of $21.15$ hours
to reach the final accepted iterate compared to $20.95$ hours for 
\ref{alg:patch_slip} and $12.74$ hours for 
\ref{alg:randomized_patch_slip}. It also shows the highest final
objective  value of $2.2529 \times 10^{-2}$ compared to
$2.2501 \times 10^{-2}$ for \ref{alg:slip} and
$2.2520 \times 10^{-2}$ for \ref{alg:patch_slip}.
Therefore, \ref{alg:slip} can be deemed as showing the worst
performance for this test case. The winner remains
somewhat inconclusive since \ref{alg:patch_slip} obtains
the better objective value but takes $62\,\%$ longer than
\ref{alg:randomized_patch_slip}. Since the remaining relative
objective difference is less than $0.1\%$, we find it fair
to conclude that \ref{alg:randomized_patch_slip} shows
a slightly better performance in this testcase.

We report the recorded running times, time to the final accepted iteration, and
the achieved objective values in \Cref{tbl:test}.
\begin{table}[h]
	\caption{Running time, time to final accepted iterate, and achieved
		objective value for \ref{alg:slip}, \ref{alg:patch_slip},
		and \ref{alg:randomized_patch_slip} in the configuration
		of \cref{sec:randomized_patch_slip}. Winners in terms of
		objective value within the reported precision
		are highlighted with pink color.}
	\label{tbl:test}
	\begin{adjustbox}{width=\textwidth}	
		\begin{tabular}{l|lll}
			\toprule
			& \multicolumn{3}{c}{\texttt{AD}} \\
			& Alg.\ S
			& Alg.\ P ($1\times 4$)
			& Alg.\ R ($1\times 4$)
			\\
			Running time [h]
			& $72$ (limit)
			& $72$ (limit)
			& $72$ (limit)
			\\
			Time to final accepted it.\ [h]
			& $32.18$
			& $46.38$
			& $14.60$
			\\						
			Objective achieved
			& \adjustbox{bgcolor=pink}{$6.7350 \times 10^{-1}$}
			& \adjustbox{bgcolor=pink}{$6.7350 \times 10^{-1}$}
			& \adjustbox{bgcolor=pink}{$6.7350 \times 10^{-1}$}
			\\			
			\midrule
			&\multicolumn{3}{c}{\texttt{CHOUPI}} \\
			& Alg.\ S
			& Alg.\ P ($4\times 1$)
			& Alg.\ R ($4\times 1$)
			\\
			Running time [h]
			& $72$ (limit)
			& $72$ (limit)
			& $72$ (limit)
			\\
			Time to final accepted it.\ [h]
			& $71.49$
			& $70.23$
			& $70.92$
			\\						
			Objective achieved
			& $3.0481 \times 10^{-3}$
			& \adjustbox{bgcolor=pink}{$3.0268 \times 10^{-3}$}
			& $3.0288 \times 10^{-3}$
			\\
			\midrule			
			&\multicolumn{3}{c}{\texttt{EXACT}} \\
			& Alg.\ S
			& Alg.\ P ($3\times 3$)
			& Alg.\ R ($3\times 3$)
			\\
			Running time [h]
			& $72$ (limit)
			& $72$ (limit)
			& $72$ (limit)
			\\
			Time to final accepted it.\ [h]
			& $3.33$
			& $19.51$
			& $62.08$
			\\						
			Objective achieved
			& $3.1196 \times 10^{-3}$
			& \adjustbox{bgcolor=pink}{$3.1022 \times 10^{-3}$}
			& $3.1042 \times 10^{-3}$
			\\			
			\midrule
			&\multicolumn{3}{c}{\texttt{HELMHOLTZ}} \\
			& Alg.\ S
			& Alg.\ P ($4\times 1$)
			& Alg.\ R ($4\times 1$)
			\\
			Running time [h]
			& $21.66$
			& $21.29$
			& $13.13$
			\\
			Time to final accepted it.\ [h]
			& $21.15$
			& $20.95$
			& $12.74$
			\\						
			Objective achieved
			& $2.2529 \times 10^{-2}$
			& \adjustbox{bgcolor=pink}{$2.2501 \times 10^{-2}$}
			& $2.2520 \times 10^{-2}$
			\\
			\bottomrule
		\end{tabular}
	\end{adjustbox}
\end{table}

\section{Conclusion}\label{sec:con}
Regarding the obtained numerical results, we draw several conclusions. Generally,
\ref{alg:slip} performs worse on our benchmark problems than 
\ref{alg:patch_slip} and \ref{alg:randomized_patch_slip}.
However, we
highlight the performance difference is nowhere near the huge
differences observed in \cite{baraldi2025domain}. The difference
between \cite{baraldi2025domain} and our work is that we are now using
the discretization of the $\TV$-seminorm from  
\cite{schiemann2025discretization} that is convergent and compatible
with discrete-valued control functions. If one accepts the 
anisotropic behavior and error of the naive discretization used in
\cite{baraldi2025domain}, the performance gain is higher. We also
note that the patch-based approaches allow for finer
discretizations since one can limit the size of the trust-region 
subproblems by choosing smaller patches; however, this may come at the cost
of slightly worse objective values. Comparing
\ref{alg:patch_slip} and \ref{alg:randomized_patch_slip}, the
performance differences are relatively small and mixed
and it is difficult to determine a universal winner. However, 
\ref{alg:randomized_patch_slip} is much more straightforward to
implement and the patch selection probabilities give an additional
degree of freedom that may turn out to have a substantial performance 
gain.

It seems also sensible us to apply a hybridization of
\ref{alg:slip} and one of the patch-based algorithms. Specifically,
we have observed that the first iterations of \ref{alg:slip} are
computationally cheap and make good progress. 
A hybridization approach capitalizes by later switching to a patch-based algorithm
when the iterations become more costly and the subproblems
of \ref{alg:slip} become intractable.

Regarding the patch layouts, rectangular layouts seem beneficial
in most cases, potentially due to subproblems
resembling to one-dimensional ones, which are generally efficiently
solvable. This has been observed and argued already in 
\cite{severitt2025integer}. How to systematically identify well-performing 
patch layouts and update them over the course of \ref{alg:patch_slip} or 
\ref{alg:randomized_patch_slip} is an important question for future research.
In addition, patch-localized adaptive finite-element methods may reduce
the computational burden further by combining the inexactness ideas from
\cite{manns2023on,antil2026afem}.

\section*{Acknowledgements}
The authors are grateful to Annika Schiemann (Greenplan GmbH, formerly TU 
Dortmund) for providing templates for several parts of the code underlying 
our computational experiments. The authors gratefully acknowledge computing
time on the LiDO3 HPC cluster at TU Dortmund, partially funded in the
Large-Scale Equipment 796 Initiative by the Deutsche Forschungsgemeinschaft
(DFG) as project 271512359.

\section*{Statements \& Declarations}
\textbf{Funding}
None. \\
\noindent
\textbf{Competing Interests}
The authors declare that they have no conflict of interest.\\
\noindent
\textbf{Author Contributions}
All authors contributed to the study conception and design. 
Material preparation, data collection and analysis were performed by 
Paul Manns. 
The first draft of the manuscript was written by 
Paul Manns and Robert Baraldi. 
All authors read and approved the final manuscript. \\
\noindent
\textbf{Data Availability} The code and data generated for the above
numerical studies are available in \href{https://github.com/paulmanns/ioc-tv-2d-benchmarks}{https://github.com/paulmanns/ioc-tv-2d-benchmarks}.

\bibliographystyle{plain}
\bibliography{references}
\appendix

\section{Asymptotics of \ref{alg:randomized_patch_slip}}\label{sec:algorithm_analysis}
In order to  analyze the iterations of \ref{alg:randomized_patch_slip}, we consider the countable stochastic
processes of random variables $(D_n)_n$ of patch decisions $D_n \in \calD$, $(W_n)_n$ of iterates,
and $(K_n)_n$ of minimum predicted reduction thresholds
$K_n \in [0,\infty)$.
Since the patch selection in l.\ \ref{ln:select_patch} is the only random operation in \ref{alg:randomized_patch_slip},
we obtain a functional dependence for the transition from $W_n$, $K_n$ to $W_{n+1}$, $K_{n+1}$,
which we denote by $T$:
\begin{gather}\label{eq:transition_function}
\begin{aligned}
T : \BVW(\Omega)\times \calD\times [0,\infty) &\to \BVW(\Omega) \times [0,\infty), \\
(W_{n+1}, K_{n+1}) &= T(W_n, K_n, D_n).
\end{aligned}
\end{gather}
For events $x \in X$, which will be defined formally in \cref{sec:probability_space} below, we thus have
$$W_{n}(x) = T(T(T(\cdots T(w_0, D_1(x)), \cdots), D_{n-2}(x)), D_{n-1}(x))$$ 
for $x \in X$ in line with \eqref{eq:transition_function}.
Under \cref{ass:standing}, every realization $\{W_n(x)\}_n$ for
$x \in X$ admits an accumulation point and our analysis will show that
every such accumulation point is necessarily stationary.

We split our analysis into several steps. We construct the underlying
probability space in \cref{sec:probability_space}. Afterwards, we 
prove the well-definedness of \ref{alg:randomized_patch_slip},
that is, the inner loop terminating, in \cref{sec:welldefinedess}.
Finally, we prove convergence in \cref{sec:convergence_proof}.

\subsection{Construction of Probability Space}\label{sec:probability_space}
First, we develop a suitable setting to define $(D_n)_n$ and a
corresponding probability space $(X, \Sigma, P)$ in a recursive
manner. We consider $X_n = \prod_{i=0}^n \calD$ and 
$\Sigma_n = \bigotimes_{i=0}^n 2^{|\calD|}$ for all $n \in \N$. We define
$\Sigma_0 = 2^{|\calD|}$ and $\mu_0$ as the probability
measure on $(X_0,\Sigma_0)$ that is implied by
$\mu_0([D_0 = D]) \coloneqq p_0^D$ for $D \in \calD$
as the base case.

Now, for a given probability measure $\mu_{n-1}$ on 
$(X_{n-1},\Sigma_{n-1})$ for $n \in \N$, we define the probability 
measure $\mu_n$ on $(X_n,\Sigma_n)$ as follows.
We define extended random variables $F_n = (F_{n,0}, \ldots, F_{n,n})$
on $X_n$ and set $D_{n} \coloneqq F_{n,n}$. We define the following 
probabilities
for $(D^0,\ldots,D^{n-1},D) \in X_{n}$
\begin{gather}\label{eq:extended_markov}
\begin{aligned}
\mu_{n}\big([F_{n,n} = D, F_{n-1,n-1} = D^{n-1}, \ldots , F_{0,0} 
= D^0]\big)
& \coloneqq p^n_D\mu_n\big([F_{n,0:n-1} = (D^0,\ldots,D^{n-1})]\big),\\
\mu_n\big([F_{n,0:n-1} = (D^0,\ldots,D^{n-1})]\big)
&\coloneqq \mu_{n-1}\big([F_{n-1} = (D^0,\ldots,D^{n-1})]\big)
\end{aligned}
\end{gather}
implying that
\[ p^n_D = \mu_n([D_{n} = D\,|\,D_{n-1} = D^{n-1},
\ldots, D_0 = D^0]),
\]
that is, that $p^n_D$ is the conditional probability
of selecting $D$ in the $n$-th iteration given the
decisions in the previous iterations. For all $n \in \N$, 
all functional dependences occur on $\sigma$-algebras 
that are power sets (of finite sets) so that they are 
measurable operations. In combination with 
\cref{ass:selection_probabilities}, we obtain that
$(p_D^n)_n$ induces a sequence of Markov kernels.
Then the Ionescu-Tulcea extension theorem implies that
there exists a unique extension to a probability
measure $P$ on $(X \coloneqq \prod_{i=0}^\infty X_i, \Sigma_P \coloneqq \bigotimes_{i=0}^\infty \Sigma_i)$ that satisfies
\begin{gather}\label{eq:ionescu_tulcea}
P\Big(A \times \prod_{i=n+1}^\infty \calD\Big) = \mu_{n}(A)
\end{gather}
for all $A \in \Sigma_n$ and all $k \in \N$.
In order to enrich the set of sets of measure zero, we consider
the completion of the constructed measure space now and,
since we only work with this one, denote it
by $(X,\Sigma_P,P)$ from now on.

We note that evaluating $T$ involves evaluating the 
(potentially multi-valued) solution maps of a sequence of 
nonsmooth and nonconvex problems in the inner loop of
\ref{alg:randomized_patch_slip}. It seems unlikely to us that
one can prove that $T$ is measurable in a straightforward
manner like verifying the prerequisites of a measurable 
selection theorem.

\subsection{Well-definedness}\label{sec:welldefinedess}
Our first step towards showing well-definedness is
that $\kappa_n(x) \searrow 0$ for all realizations $x \in X$.
\begin{lemma}\label{lem:kappa_to_zero}
$X = \big\{ x \in X\,:\, \kappa_n(x) \searrow 0\big\}$.
\end{lemma}
\begin{proof}
By way of contradiction, let $x \in X\setminus \{ x \in X\,:\, \kappa_n(x) \searrow 0\}$, that is, $\kappa_n(x) \searrow \kappa > 0$.
Then the condition in \ref{alg:randomized_patch_slip} l.\ \ref{ln:reduce_kappa} is evaluated to true only finitely
many times and there is thus a largest outer iteration $n_0$ such that for all $n \ge n_0$ and all inner iterations
$k$, the condition in Line \ref{ln:psuff_dec} is satisfied or (else) the condition in Line \ref{ln:reduce_kappa}
is evaluated to false.
The first case can only happen finitely many times because otherwise the actual reduction would be bounded below
by $\sigma \kappa$ and thus $J(W_n(x)) \to -\infty$ holds, which contradicts that $F$, $\TV$ and in turn $J$
are bounded below.
Consequently, there exists an outer iteration $n$ such that for all inner iterations Line \ref{ln:psuff_dec} is
never satisfied and $\kappa_n$ is never reduced. 

In outer iteration $n$, the inner loop does not terminate so that $\Delta_0 2^{-k} \to 0$, which implies
$\|\tilde{W}_{k,n,D_n}(x) - W_n(x)\|_{L^1} \to 0$. By virtue of \cref{ass:standing}, we obtain
\[ F(W_n(x)) - F(\tilde{W}_{k,n,D_n}(x)) - (\nabla F(W_n(x)), W_n(x) - \tilde{W}_{k,n,D_n}(x))_{L^2} \to 0
\]
and in turn $\ared^{k,n,D_n} \to \pred^{k,n,D_n} > \kappa_n \ge \kappa > 0$ so that
$\ared^{k,n,D_n} \ge \sigma \pred^{k,n,D_n}$ holds eventually, which contradicts 
that Line \ref{ln:psuff_dec} is never satisfied.
\end{proof}
Next, we deduce that the inner loop of \ref{alg:randomized_patch_slip} that starts in Line \ref{ln:tr_reduction} terminates finitely.
\begin{lemma}\label{lem:well-defined_inner_loop}
The inner loop starting in \ref{alg:randomized_patch_slip} l.\ \ref{ln:tr_reduction} is well defined
for all $x \in X$, that is, it terminates finitely.
\end{lemma}
\begin{proof}
This follows directly from \cref{lem:kappa_to_zero} and the fact that finite termination of the inner loop
is necessary for a reduction of $\kappa_n(x)$.
\end{proof}

\subsection{Convergence Proof}\label{sec:convergence_proof}
Before proving our results formally, we give a brief intuition about
what happens here when the algorithm hits a point that is patch-stationary
on $D$. For the case that $D_n = D$, $W_n = w$, $K_n = \kappa$ and $w$
is patch-stationary on $D$, it can happen that a step is
accepted nevertheless in the inner loop for some large enough $k \in \N$. In this case, the fact that the objective values decrease monotonically
over the iterations of \ref{alg:randomized_patch_slip} implies that the algorithm moves away from $w$ and cannot
return so that this case is not important to our analysis. If no
step is accepted, either $\pred^{n,k,D} \searrow 0$ as $k \to \infty$ or $\pred^{n,k,D} = 0$ for some large enough $k$. Thus, eventually
$\pred^{n,k,D} \le \kappa_n$ holds so that for the relevant iterates 
that are patch-stationary on $D$, we obtain
\[ T(W_n, K_n, D_n) = (W_n, 2^{-1} K_n). \]
Before finally proving convergence to stationary points
with probability one, we need the auxiliary result
that all weak-$^*$ accumulation points of all realizations are strict.
\begin{lemma}\label{lem:strict_convergence}
It holds that
\[ N_1 \coloneqq \big\{ x \in X\,:\,
\exists (n_\ell)_\ell \subset (n)_n : W_{n_\ell}(x) \weakstarto \bar{w} \text{ in } \BVW(\Omega)
\text{ and } \bar{w} \text{ is not a strict limit} \big\} \in \Sigma_P
\]
and $P(N_1) = 0$.
\end{lemma}
\begin{proof}
Let $x \in N_1$ with corresponding subsequence $(n_\ell)_\ell \subset (n)_n$ and weak-$^*$ but not strict limit
$\bar{w} \in \BVW(\Omega)$. Since $F$ is continuous and bounded below and the objective $J$ is monotonically
decreasing over the iterations by virtue of the acceptance criterion in \ref{alg:randomized_patch_slip} l.\ \ref{ln:psuff_dec}
it follows that the sequence $\{J(W_{n_\ell}(x))\}_\ell$ converges
and there is no subsequence of $\{W_{n_\ell}(x)\}_\ell$ that converges strictly to $\bar{w}$.

A close inspection of (the proof of) Lemma 5.7 in \cite{baraldi2025domain} gives that there exist $D \in \calD$,
$\ell_0 \in \N$, and $k_0 \in \N$, such that $\ared^{n_\ell,k_0,D} \ge \sigma \pred^{n_\ell,k_0,D} > p > 0$
for some $p > 0$ and all $\ell \ge \ell_0$. Since $\kappa_n \searrow 0$ holds, we can increase $\ell_0$ and
obtain that \ref{alg:randomized_patch_slip} l.\ \ref{ln:psuff_dec} is satisfied for all $\ell \ge \ell_0$ if
$D_{n_\ell}(x) = D$. Consequently, we obtain $J(W_{n}(x)) \to -\infty$ holds if the
choice $D_{n_{\ell_m}}(x) = D$ is made for infinitely many $\ell$ in \ref{alg:randomized_patch_slip}
l.\ \ref{ln:select_patch}.

This contradicts that $F$, $\TV$, and in turn $J$ are bounded below
and we obtain $x \in N_1^D$, where
\[ N_1^D = 
\big\{ x \in X\,:\,
|\{\ell \in \N\,:\, D_{n_\ell}(x) = D\}| < \infty \big\}.
\]
We can write $N_1^D = \bigcup_{k=1}^\infty N_1^D(k)$ with
\[ N_1^D(k) = 
\big\{ x \in X\,:\, \max \{ \ell \in \N\,:\, D_{n_\ell}(x) = D \} \le k \big\}
= \bigcap_{\ell \ge k}
\big\{ x\in X\,:\,D_{n_\ell}(x) \in \calD \setminus \{D\}\big\},
\]
where $N_1^D(k)$ is measurable as the countable intersection of measurable sets and
$N_1$ is measurable as a countable union of measurable sets.

Because the trust-region radius shrinks for increasing inner iterations, we have 
$\pred^{n_\ell,0,D} \ge \pred{n_\ell,k_0,D} > p > 0$ for all $\ell \ge \ell_0$.
Since $\kappa_n \searrow 0$ holds in \cref{ass:selection_probabilities},
we obtain (after potentially increasing $\ell_0$) that $ p^{n_\ell}_D \ge p_0$
holds for all $\ell \ge \ell_0$.

Consequently, we have for all $\ell \ge \ell_0$ that
\begin{gather}\label{eq:probability_except_D}
\mu_{n_\ell}\big(\big\{ x\in X\,:\,D_{n_\ell}(x) \in \calD \setminus \{D\} \big\}\big) = 1 - p^{n_\ell}_D \le 1 - p_0
\end{gather}
holds. We deduce by means of the regularity of $P$ that
\begin{align*}
P(N_1^D(k)) &\le P\Big(\bigcap_{\ell \ge \max\{k,\ell_0\}}
\big\{ x\in X\,:\,D_{n_\ell}(x) \in \calD \setminus \{D\}\big\}\Big) \\
&= \lim_{m\to\infty} P\Big(\bigcap_{\ell = \max\{k,\ell_0\}}^m
\big\{ x\in X\,:\,D_{n_\ell}(x) \in \calD \setminus \{D\}\big\}\Big)  \\
&= \lim_{m\to\infty} P\Big(\prod_{i=1}^\infty A_i^m\Big) \\
&= \lim_{m\to\infty} \prod_{\ell = \max\{k,\ell_0\}}^m (1 - p_{n_\ell}^D)
\le \lim_{m\to\infty} (1 - p_0)^m = 0,
&& \eqref{eq:extended_markov},\eqref{eq:ionescu_tulcea},\eqref{eq:probability_except_D}
\end{align*}
where
\[ A_n^m = \left\{
\begin{aligned}
\calD \setminus \{D\} & \text{ if }
n = n_\ell \text{ for some } \ell \in \{\max\{k,\ell_0\},\ldots,m\}
\\
\calD & \text{ else.}
\end{aligned}
\right.
\]
Employing the regularity of $P$ again gives $P(N_1^D) = 0$ and in turn that $N_1 \in \Sigma_P$
with $P(N_1) = 0$ because $N_1^D \subset N_1$.
\end{proof}
We are ready to prove our main result that
\ref{alg:randomized_patch_slip} converges
to a stationary point with probability
one. We briefly recall the theorem.
\mainthm*
\begin{proof}
By construction of $P$, every set that is a subset of a measurable set with measure zero is measurable
and has measure zero. Therefore, it is sufficient to show that $N$ is included in a measurable set of measure
zero. 

To prove the claim, we now follow the 
deterministic strategy of Theorem 6.4 in 
\cite{manns2023on} and Theorem 5.8 in 
\cite{baraldi2025domain}. From 
\cref{lem:strict_convergence},
we obtain strict convergence with probability 
one. Thus, it remains to prove $N_2 \in \Sigma_P$ and $P(N_2) = 0$ for
\[ N_2 \coloneqq
\big\{ x \in X\,:\,
\exists (n_\ell)_\ell \subset (n)_n : W_{n_\ell}(x) \to \bar{w} 
\text{ strictly in } \BVW(\Omega)
\text{ for some instationary } \bar{w}
\big\}
\]
and the claim follows from $N \subset N_2$.
Let $x \in N_2$ with corresponding subsequence $(n_\ell)_\ell \subset (n)_n$ and strict limit $\bar{w} \in \BVW(\Omega)$.
Applying Lemma 5.3 in \cite{baraldi2025domain} gives the existence of $D \in \calD$, $k_0 \in \N$, $\ell_0 \in \N$,
and $p > 0$ such that
\[ {\ared}^{n_\ell,k_0,D} \ge \sigma {\pred}^{n_\ell,k_0,D}
\enskip\text{and}\enskip
{\pred}^{n_\ell,k_0,D} > p \]
hold for all $\ell \ge \ell_0$. After potentially increasing $\ell_0$, we also obtain
$p > \kappa_{n_\ell}$ 
and, because of
$\pred^{n_\ell,0,D}
\ge \pred^{n_\ell,k_0,D}$,
$\pred^{n_\ell,0,D} > \varepsilon_{n_\ell}$
for all $\ell \ge \ell_0$. Consequently,
\eqref{eq:probability_except_D} holds
for $\ell \ge \ell_0$
as in \cref{lem:strict_convergence}.

This means that the sufficient decrease condition \ref{alg:randomized_patch_slip} l.\ \ref{ln:psuff_dec} is satisfied
and we obtain $J(W_{n_\ell}(x)) \to -\infty$ if $D_{n_\ell}(x) = D$ holds for infinitely many
$\ell \ge \ell_0$. This contradicts that $F$, $\TV$, and in turn $J$ are bounded below
so that $x \in N_2^D =  \big\{ x \in X\,:\, |\{\ell \in \N\,:\, D_{n_\ell}(x) = D\}| < \infty \big\}$.
As in \cref{lem:strict_convergence}, we can write $N_2^D = \bigcup_{k=1}^\infty N_2^D(k)$ with
$N_2^D(k) = \bigcap_{\ell \ge k} \big\{ x\in X\,:\,D_{n_\ell}(x) \in \calD \setminus \{D\}\big\}$.

With the same arguments as in 
\cref{lem:strict_convergence}
we obtain $N_2^D(k) \in \Sigma_P$,
$P(N_2^D(k)) \le \lim_{m\to\infty} (1 - p_0)^m = 0$ from \eqref{eq:probability_except_D},
and in turn $N_2 \in \Sigma_P$
and $P(N_2) = 0$.
\end{proof}

% \section{Computational performance enhancement
	% on trust-region algorithm
% }
%

\section{\texorpdfstring{($\operatorname{TRP}^h$) $\Gamma$-converges to ($\operatorname{\ref{eq:trp}}$)}{Gamma Convergence. }}\label{sec:trph_trp}
In the setting of \cref{sec:w_tvh_discretization}
and for $g \in L^1(D)$, we define
the functionals $j^h : DG0^{\tau_h} \to \R\cup\{\infty\}$,
$j : L^\infty(D) \to \R$ as 
\begin{align*}
j^h(w) &\coloneqq 
   \int_D gw \dd x +
   + \alpha \underbrace{\max\bigl\{\tfrac{1}{\sqrt{2}} \TV(w), \TV^h(w)\bigr\}}_{\eqqcolon T^h(w)}
   + \delta_{W}(w) + \delta_{\Delta}(w), \\
j(w)  &\coloneqq \int_\Omega gw \dd x +
+ \alpha \TV(w)
+ \delta_{W}(w) + \delta_{\Delta}(w),
\end{align*}
where $\delta_{W}$ is the $\{0,\infty\}$-valued
indicator functional of the set
$\calW \coloneqq \{ w \in L^\infty(D) : w(x) \in W \text{ for a.e.\ } x \in D\}$
and $\delta_\Delta$ is the $\{0,\infty\}$-valued indicator
functional of the set
$\{ w \in L^\infty(D) : \|w - \bar{w}\|_{L^1}(D) \le \Delta\}$.
Thus, we have restricted w to the domain $D$ of 
\eqref{eq:trp}, on which changes can take place and tacitly
assume that $w$ is extended by $\bar{w}$ outside on $\Omega\setminus D$. We argue the intended $\Gamma$-convergence and convergence
of minimizers below.
\begin{lemma}\label{lem:trph_trp}
Let $\tau_h \in o(h)$ as $h \searrow 0$. Then
$j^h$ $\Gamma$-converges to $j$ with respect to
weak-$^*$ convergence in $\BV(\Omega)$
on the set $\BVW(\Omega)$.
\end{lemma}
\begin{proof}
The $\liminf$-inequality follows as in 
\cite{schiemann2025discretization} since
the only additional term $w \mapsto \delta_{\Delta}(w)$
in the functional is
lower semi-continuous due to $w \mapsto \|w - \bar{w}\|_{L^1}$
being continuous with respect to weak-$^*$ convergence
in $\BV(\Omega)$.

It remains to show the $\limsup$-inequality, which means that
for a given $w \in \BVW(\Omega)$, we need to construct 
$w^{h} \in DG0^{\tau_h}(D) \cap \calW$ such that
\[ w^{h} \to w \text{ in } L^1(\Omega),
   \quad
   T^h(w^{h}) \to \TV(w) \text{ in } \R,
   \quad
   \|w^{h} - \bar{w}\|_{L^1(D)} \le \Delta
\]
as $h \searrow 0$ with $\tau_h \in o(h)$
if $j(w) < \infty$. If $j(w) = \infty$, any sequence
will do. Thus, if $\|w - \bar{w}\|_{L^1(D)} > \Delta$,
there is no need further argument needed because $j(w) = \infty$.
If $\|w - \bar{w}\|_{L^1(D)} < \Delta$, the sequence from Theorem
3.13 from \cite{schiemann2025discretization} can be used
since $\|w^{h} - \bar{w}\|_{L^1(D)} \le \Delta$ holds eventually
in this case.

It remains to consider the case $\|w - \bar{w}\|_{L^1} = \Delta$.
In this case, we may employ the construction from the
proof of Theorem 5.2 from 
\cite{manns2023on}---see case $\|w - v\|_{L^1} = \Delta$ therein---to 
deduce the existence of a sequence $w^n \weakstarto w$ in $\BV(\Omega)$
with $\TV(w^n) \to \TV(w)$ with a corresponding sequence
$\varepsilon_n \to 0$ such that
$\|w^n - w\|_{L^1} \le \varepsilon_n/2$ and
$\|w^n - \bar{w}\|_{L^1} < \Delta - \varepsilon_n/2$.
For every $n \in \N$, we consider $(w^n)^h$ for $w^n$ from
Theorem 3.13 in \cite{schiemann2025discretization}.
Thus there always exists $h_n$ such that for $h \in (0,h_n)$,
we have $\|(w^n)^h - w^n\| \le \varepsilon_n / 2$,
$|T^h((w^n)^h) - \TV(w^n)| < \varepsilon_n / 2$
implying $\|(w^n)^h - w\|_{L^1} < \varepsilon_n$
and $\|(w^n)^h - \bar{w}\| \le \Delta$.
By choosing $w^h \coloneqq (w^n)^h$ for $h \in [h_{n+1},h)$, we obtain
the desired properties and the $\limsup$-inequality follows.
\end{proof}
As a consequence, we obtain convergence of minimizers below.
\begin{theorem}
Let $w^h$ solve $\operatorname{TRP}^h(\bar{w}, g, D, \Delta)$.
Then the sequence $(w^h)_h$ admits an accumulation
point with respect to weak-$^*$-convergence in $\BV(\Omega)$
and every such accumulation point solves
$\operatorname{\ref{eq:trp}}(\bar{w}, g, D, \Delta)$.
\end{theorem}
\begin{proof}
In addition to the $\Gamma$-convergence result
from \cref{lem:trph_trp}, we require boundedness
in $\BVW(\Omega)$. This follows from
the construction of $T^h$ by virtue of Lemma 3.9
in \cite{schiemann2025discretization}.
Then the result follows with the ususal argument
to show convergence of minimizers for $\Gamma$-converging
functionals; see, e.g., Theorem 1.21 in
in \cite{braides2002gamma}.
\end{proof}

\end{document}